\documentclass[<opyions>]{elsarticle}
\usepackage{amssymb}
\usepackage{graphicx}
\usepackage{epstopdf}
\usepackage{subfigure}
\usepackage{extarrows}
\usepackage{fancyhdr}
\usepackage{amsthm}
\usepackage{extarrows}
\usepackage{caption}
\usepackage{float}
\usepackage{algorithm}
\usepackage{algorithmic}
\usepackage{multirow}
\usepackage{amsmath}
\usepackage[numbers]{natbib}
\usepackage{graphicx}
\usepackage{amsmath} 
\usepackage{amsfonts} 
\usepackage{mathrsfs} 
\usepackage{booktabs} 
\usepackage{color} 

\usepackage{color}
\usepackage{hyperref}

\usepackage[toc,page,title,titletoc,header]{appendix}
\usepackage{appendix}
\theoremstyle{definition}

\theoremstyle{definition}
\usepackage{hyperref}
\usepackage{latexsym, bm}
\usepackage{fancyhdr}
\usepackage{mathrsfs}
\usepackage{wasysym}
\usepackage{float}
\usepackage{titlesec}
\usepackage{pgfplots}
\usepackage{tikz}
\usepackage{subfigure}
\usepackage{natbib}
\biboptions{numbers,sort&compress}

\usetikzlibrary{arrows,shapes,positioning}
\usetikzlibrary{decorations.markings}
\tikzstyle arrowstyle=[scale=1]
\tikzstyle directed=[postaction={decorate,decoration={markings,
    mark=at position .65 with {\arrow[arrowstyle]{stealth}}}}]
\tikzstyle reverse directed=[postaction={decorate,decoration={markings,
    mark=at position .65 with {\arrowreversed[arrowstyle]{stealth};}}}]

\newtheorem{lemma}{Lemma}[section]
\newtheorem{theorem}{Theorem}[section]

\newtheorem{example}{Example}[section]

\allowdisplaybreaks
\begin{document}
	
\begin{frontmatter}
\title{{\bf A semi-implicit DLN Galerkin finite element method for coupled Ginzburg-Landau equations with general nonlinearity}}

\author{Zhen Guan\corref{cor1}}
\ead{zhenguan1993@foxmail.com}

\author{Xianxian Cao}
\author{Junjun Wang}

\cortext[cor1]{Corresponding author.}

\address{School of Mathematics and Statistics, Pingdingshan University, Pingdingshan, 467000, China}

\begin{abstract}
In this paper, based on the two-step discretization scheme proposed by Dahlquist, Liniger and Nevanlinna (DLN), we develop a semi-implicit Galerkin finite element method for solving the coupled generalized Ginzburg-Landau equations. By virtue of a novel analytical technique, the boundedness of the numerical solution in the infinity norm is established, upon which the unconditionally optimal error estimates in the $L^2$ and $H^1$-norms are further derived. Compared with the space-time error splitting technique commonly adopted in the literature, the analytical method proposed in this paper does not require the introduction of an additional temporal discretization system, thus greatly simplifying the theoretical argument. The core point of the argument lies in the skillful application of the inverse inequality and discrete Agmon inequality to analyze the two cases, namely $\tau \leq h$ and $\tau>h$, respectively. Three numerical examples covering both two- and three-dimensional scenarios are eventually provided for the validation of the theoretical findings.
\end{abstract}

\end{frontmatter}

\thispagestyle{empty}

\numberwithin{equation}{section}
\section{Introduction}\label{section01}
In this paper, we propose and analyze an efficient numerical method to solve the following coupled Ginzburg-Landau equations involving a general nonlinear term:
\begin{align}
&u_t - (\nu_1 + \mathrm{i}\alpha_1)\Delta u + \left((\kappa_1 + \mathrm{i}\beta_1)f_1(\vert u \vert^2)+(\mu_1 + \mathrm{i}\delta_1)g_1(\vert v \vert^2)\right) u - \gamma_1 u = 0, \quad (\boldsymbol{x},t) \in \Omega \times (0,T],\label{202512272133} \\
&	v_t - (\nu_2 + \mathrm{i}\alpha_2)\Delta v + \left((\kappa_2 + \mathrm{i}\beta_2)f_2(\vert u \vert^2)+(\mu_2 + \mathrm{i}\delta_2)g_2(\vert v \vert^2)\right) v - \gamma_2 v = 0, \quad (\boldsymbol{x},t) \in \Omega \times (0,T],\label{202512272134} \\
&u(\boldsymbol{x},t) = 0, \quad v(\boldsymbol{x},t) = 0, \quad (\boldsymbol{x},t) \in \partial\Omega\times(0,T], \label{202512272135}\\
&u(\boldsymbol{x},0) = u^0(\boldsymbol{x}),\quad v(\boldsymbol{x},0) = v^0(\boldsymbol{x}), \quad \boldsymbol{x} \in \bar{\Omega},\label{202512272136}
\end{align}
where $i=\sqrt{-1}$ is the imaginary unit, $u(\boldsymbol{x},t),v(\boldsymbol{x},t)$ are complex-valued functions, $\Omega\subset \mathbb{R}^d~(d=2,3)$ is a bounded convex polytopal domain, $\partial\Omega$ denotes the boundary of $\Omega$, $\nu_1>0,\nu_2>0,\kappa_1>0,\kappa_2>0,\mu_1>0,\mu_2>0,\alpha_1,\alpha_2,\beta_1,\beta_2,\delta_1,\delta_2$ are given real constants, $f_1(\cdot),f_2(\cdot),g_1(\cdot),g_2(\cdot)$ are non-negative real functions that are continuously differentiable of the first order and defined on the infinite right half-interval $[0,+\infty)$, and  $u^0(\boldsymbol{x})$ and $v^0(\boldsymbol{x})$ are known complex-valued functions satisfying the consistency condition $u^0(\boldsymbol{x})|_{\partial\Omega}=v^0(\boldsymbol{x})|_{\partial\Omega}=0$.

As a class of complex-valued nonlinear partial differential equations, the Ginzburg-Landau equation plays a fundamental role in such diverse fields as low-temperature superconductivity, nonequilibrium hydrodynamics, and nonlinear optics. For further details on the physical applications and mathematical analyses of this equation, readers are referred to the recent monograph on the subject \cite{GuoJiangLi2020}. As analytical solutions to nonlinear models are generally not tractable, a large number of scholars have developed efficient numerical algorithms for solving the single Ginzburg-Landau equation in recent decades. By way of illustration, we refer to the most recent relevant studies.
Zhang et al. \cite{ZhangYangWang2025} presented a linearized three-layer finite difference (TLFD) scheme for the two-dimensional complex Ginzburg-Landau equation characterized by strong nonlinearities. Using the energy analysis method and mathematical induction method, they derived the unique solvability and unconditional convergence of the numerical scheme. Li et al. \cite{LiCuiZhang2025} proposed a Crank-Nicolson element-free Galerkin method for solving the Ginzburg-Landau equation with a cubic nonlinear term, and proved the optimal error estimate of the fully discrete numerical scheme by means of the space-time error splitting technique.
Caliari and Cassini \cite{CaliariCassini2024} developed an efficient numerical scheme for solving the two- and three-dimensional complex Ginzburg-Landau equations with cubic or cubic-quintic nonlinearities, where the high-order exponential splitting method was employed for temporal discretization and the Fourier spectral method for spatial discretization. Wang and Li \cite{WangLi2023} studied a linearized variable-time-step BDF2 virtual element method for solving the nonlinear Ginzburg-Landau equation, carried out rigorous unconditional error analysis using DOC/DCC kernels and error splitting technique, and verified the method's efficiency and stability through numerical experiments on rectangular and Voronoi-Lloyd polygons. Gao and Xie \cite{GaoXie2023} discussed the numerical analysis of a finite element method for the time-dependent Ginzburg-Landau equation under the Coulomb gauge. At the discrete stage, the main goal of this work was to derive the second-order spatial convergence of the primary variable $\psi$, in contrast to the $O(h)$ spatial accuracy of $A$'s numerical solutions. Further details on the development and applications of numerical methods can be found in the literature \cite{GaoLiSun2014,DuanZhang2022,MaQiao2023,PeiWeiZhang2024,GuanCao2025} and the references contained therein. 

Although extensive research has been conducted on a single Ginzburg-Landau equation, studies focusing on the coupled Ginzburg-Landau equations (including fractional-order and nonlocal coupled models) remain relatively limited. For example, Li et al. \cite{LiWang2023} addressed the coupled Ginzburg-Landau equations via a temporal linearized variable-step second-order backward differentiation formula integrated with a spatial nonconforming virtual element method. 
Xu et al. \cite{XuZengHu2019} numerically investigated the space fractional coupled Ginzburg-Landau equations by utilizing a linearized semi-implicit difference scheme. They also established the unique solvability, stability, and convergence of the numerical algorithm. Li et al. \cite{LiHuang2019} proposed an efficient finite difference scheme for solving the fractional Ginzburg-Landau equations involving the fractional Laplacian, which is constructed by combining an implicit time-integration method with the weighted and shifted Grünwald difference method. Recently, Ding et al. \cite{DingZhangYi2025} conducted exhaustive mathematical analysis and numerical simulations on the space fractional Ginzburg-Landau equations. 

In \cite{DahlquistLinigerNevanlinna1983}, Dahlquist, Liniger, and Nevanlinna put forward a single-parameter family of one-leg, two-step numerical schemes--referred to here as the DLN method. This approach boasts G-stability (a form of nonlinear stability), and maintains second-order accuracy regardless of the chosen time grid. Based on this framework, we construct a semi-implicit fully discrete Galerkin finite element method in this work. To perform a theoretical analysis of the proposed linearized numerical scheme, the argumentation approach commonly adopted in the literature is to decompose the error into temporal and spatial components by introducing a time-discrete system. This method was first proposed by Buyang Li in his doctoral dissertation in 2012 \cite{Li2012}, and it was named the space-time error splitting technique. It has since been applied to a variety of complex nonlinear coupled partial differential systems \cite{Gao2016,GaoSun2021,CaiLiChen2018,YangJiang2021,LiWangGuan2025}. However, the introduction of an additional time-discrete system complicates the argumentation to a certain extent. Different from the existing literature, this paper proposes a novel proof method to establish the boundedness of the numerical solutions in the infinity norm, which constitutes a key innovation of this work. The core of the proof in this paper lies in the case analysis of $\tau$ and $h$, namely: when $\tau\leq h$, we prove the $L^{\infty}$-boundedness of the numerical solutions by virtue of the inverse inequality and the Poincaré inequality, while when  $\tau > h$, we derive this property by employing the discrete Agmon inequality.
The numerical analysis method presented in this paper does not require the introduction of an additional temporal discretization system, which greatly simplifies the theoretical deductions.

The structure of this paper is organized as follows. In Section \ref{section2}, we formulate the fully discrete DLN finite element algorithm constructed in this paper, as well as the lemmas necessary for theoretical analysis. Section \ref{section3} is devoted to the theoretical analysis of the numerical scheme, including its stability, the boundedness of the numerical solutions in the infinity norm, and the unconditionally optimal error estimate. In Section \ref{section4}, three numerical examples are provided to verify the effectiveness of the numerical algorithm. The summary and outlook of this paper are presented in Section \ref{section5}.

\section{Derivation of the fully discrete DLN numerical scheme}\label{section2}
\subsection{The variational formulation}
Before presenting the weak formulation of the equations, we first introduce the relevant notations for Sobolev spaces. Throughout this paper, denote by \(L^2(\Omega)\) the space of complex-valued square-integrable functions, and its inner product and norm are defined as follows:
\begin{equation*}
(u,v) = \int_{\Omega}uv^*\text{d}\boldsymbol{x},\quad \|u\|=\sqrt{(u,u)},\quad u,v\in L^2(\Omega), 
\end{equation*}
where $v^*$ represents the conjugate of $v$. Moreover, let $\|\cdot\|_{\infty}$ be the essential supremum norm in the Lebesgue space $L^{\infty}(\Omega)$.  For each positive integer $m$, we denote by \(H^m(\Omega)\) the standard complex-valued Sobolev space consisting of functions whose weak derivatives up to order $m$ belong to \(L^2(\Omega)\), equipped with the norm \(\|\cdot\|_m\). \(H_0^1(\Omega)\) is the space consisting of functions in \(H^1(\Omega)\) that have zero boundary traces. Besides, we recall the Bochner space $L^{p}(0,T;X)$, equipped with the norm
\begin{equation*}
\|u\|_{L^{p}(0,T;X)} = \left(\int_{0}^{T}\|u\|^p_{X}\text{d}t\right)^{\frac{1}{p}},\quad 1\leq p< \infty,
\end{equation*}
and 
\begin{align*}
\|u\|_{L^\infty(0,T;X)} = \operatorname{ess} \sup_{t\in[0,T]} \|u\|_X,
\end{align*}
where $X$ is a complex Banach space. In this paper, the constant $C$ may take different values at different occurrences.

From the notations introduced above, we can immediately obtain the weak form of equations \eqref{202512272133}-\eqref{202512272136} as: Find $u, v\in L^{2}(0,T;H_0^1(\Omega))\cap H^{1}(0,T;L^2(\Omega))$ such that 
\begin{align}
&(u_t,\phi)+(\nu_1 + \mathrm{i}\alpha_1)(\nabla u,\nabla\phi) + \left(\left((\kappa_1 + \mathrm{i}\beta_1)f_1(\vert u \vert^2)+(\mu_1 + \mathrm{i}\delta_1)g_1(\vert v \vert^2)\right) u,\phi\right) - \gamma_1 (u,\phi) = 0, \label{202512291102}\\
&(v_t,\psi) + (\nu_2 + \mathrm{i}\alpha_2)(\nabla v,\nabla\psi) + \left(\left((\kappa_2 + \mathrm{i}\beta_2)f_2(\vert u \vert^2)+(\mu_2 + \mathrm{i}\delta_2)g_2(\vert v \vert^2)\right)v,\psi\right) - \gamma_2 (v,\psi) = 0,\label{202512291103}
\end{align}
where $\phi$ and $\psi$ are arbitrary functions in $H_0^{1}(\Omega)$ and $u(0)=u^0, v(0)=v^0$.
\subsection{Finite element terminology}
Let $\mathcal{T}_h=\{K\}$ be a sequence of quasi-uniform simplicial partitions of the domain \(\Omega\). $h$ is the spatial mesh size, which is defined as the maximum value of the diameter of each element. Then, we introduce the $k$-th order Lagrange finite element space as follows:
\begin{align*}
V_h^{k} = \{v_h\in H_{0}^{1}(\Omega):v_h|_K\in\mathbb{P}_k(K)\}, 
\end{align*}
where $\mathbb{P}_k(K)$ is the space of polynomials defined on the element \(K\) with degree at most $k$.

let $R_h:H_{0}^{1}(\Omega)\rightarrow V_h^{k}$ the Ritz projection operator defined by
\begin{align*}
(\nabla R_h v, \nabla v_h) = (\nabla v, \nabla v_h),\quad v \in H_0^1(\Omega),\quad \forall v_h \in V_h^k.
\end{align*}
As is well known, according to the classical finite element theory established in the early 1970s, the elliptic projection operator satisfies the following properties: 
\begin{align*}
\| v-R_h v \| + h\| \nabla(v - R_h v) \| \leq C h^{k+1} \| v \|_{k+1}, \quad \forall v \in H_0^1(\Omega) \cap H^{k+1}(\Omega).
\end{align*}
By virtue of the quasi-uniformity of the mesh partition, it is known that the inverse inequality holds
\begin{align}
\|v_h\|_{\infty} \leq C h^{-\frac{d}{2}} \|v_h\|, \quad \forall v_h \in V_h^k, \quad d=2,3.\label{202512302043}
\end{align}

We next introduce the the discrete Laplacian $\Delta_h:V_h^k\rightarrow V_h^k$ by 
\begin{align*}
-(\Delta_hu_h,v_h) = (\nabla u_h,\nabla v_h),\quad u_h\in V_h^k,\quad \forall v_h\in V_h^k. 
\end{align*}
Following the same line of reasoning as in the proof of Lemma 3.4 in \cite{Kirk2024}, the discrete Agmon inequality holds 
\begin{align}
\| v_h \|_{\infty} \leq C\| \nabla v_h \|^{1/2} \| \Delta_h v_h \|^{1/2}, \quad \forall v_h \in V_h^k. \label{202512302053}
\end{align}
\subsection{Numerical method}
Let $0 = t_0 < t_1 < \dots < t_N = T$ be a uniform partition of the time interval $[0, T]$ with the time step size $\tau=\frac{T}{N}$. For every real-valued constant $\theta \in [0,1]$, denote $t_{n+\frac{\theta}{2}}=\left(n+\frac{\theta}{2}\right)\tau$. For a sequence of functions $\{u^n\}_{n=0}^N$ defined on the time grid (where $u^n = u(t_n)$ denotes the value of function $u$ at time $t_n$), we introduce the following difference quotient notations:
\begin{align*}
&D_\tau u^1=\frac{u^1-u^0}{\tau},\quad D_\tau u^n = \frac{(1+\theta)u^n-(2\theta) u^{n-1}+(\theta-1)u^{n-2}}{2\tau},\quad 2\leq n \leq N,\\
&\hat{u}^1=\frac{u^1+u^0}{2},\quad \hat{u}^n=\frac{(2+\theta-\theta^2)u^n+(2\theta^2)u^{n-1}+(2-\theta-\theta^2)u^{n-2}}{4},\quad 2\leq n \leq N,\\
& \tilde{u}^n=\left(1+\frac{\theta}{2}\right)u^{n-1}-\frac{\theta}{2}u^{n-2},\quad 2\leq n \leq N.
\end{align*}
Next, we are ready to present the fully discrete finite element algorithm: Find $u_h^n,v^n_h\in V_h^k~(n\geq2),$ such that 
\begin{align}
	&(D_\tau u_h^n,\phi_h)+(\nu_1 + \mathrm{i}\alpha_1)(\nabla \hat{u}_h^n,\nabla\phi_h) + (\kappa_1 + \mathrm{i}\beta_1)(f_1(\vert \tilde{u}^n_h \vert^2)\hat{u}_h^n,\phi_h)\notag\\
	&\quad + (\mu_1 + \mathrm{i}\delta_1)(g_1(\vert \tilde{v}^n_h \vert^2)\hat{u}_h^n,\phi_h) - \gamma_1 (\hat{u}_h^n,\phi_h) = 0, \quad \forall \phi_h \in V_h^{k},\label{12281754}\\
	&(D_\tau v_h^n,\psi_h) + (\nu_2 + \mathrm{i}\alpha_2)(\nabla \hat{v}_h^n,\nabla\psi_h) + (\kappa_2 + \mathrm{i}\beta_2)(f_2(\vert \tilde{u}_h^n \vert^2)\hat{v}_h^n,\psi_h)\notag\\
	&\quad + (\mu_2 + \mathrm{i}\delta_2)(g_2(\vert \tilde{v}^n_h  \vert^2)\hat{v}_h^n,\psi_h) - \gamma_2 (\hat{v}_h^n,\psi_h) = 0,\quad \forall\psi_h\in V_h^{k},\label{12281755}
\end{align}
and $u_h^1,v_h^1$ is determined by the linearized Crank-Nicolson scheme 
\begin{align}
&(D_\tau u_h^1,\phi_h)+(\nu_1 + \mathrm{i}\alpha_1)(\nabla \hat{u}_h^1,\nabla\phi_h) + (\kappa_1 + \mathrm{i}\beta_1)(f_1(\vert u^\frac{1}{2} \vert^2)\hat{u}_h^1,\phi_h)\notag\\
&\quad + (\mu_1 + \mathrm{i}\delta_1)(g_1(\vert v^\frac{1}{2} \vert^2)\hat{u}_h^1,\phi_h) - \gamma_1 (\hat{u}_h^1,\phi_h) = 0, \quad \forall \phi_h \in V_h^{k},\label{12281756}\\
&(D_\tau v_h^1,\psi_h) + (\nu_2 + \mathrm{i}\alpha_2)(\nabla \hat{v}_h^1,\nabla\psi_h) + (\kappa_2 + \mathrm{i}\beta_2)(f_2(\vert u^\frac{1}{2} \vert^2)\hat{v}_h^1,\psi_h)\notag\\
&\quad + (\mu_2 + \mathrm{i}\delta_2)(g_2(\vert v^\frac{1}{2}  \vert^2)\hat{v}_h^1,\psi_h) - \gamma_2 (\hat{v}_h^1,\psi_h) = 0,\quad \forall \psi_h\in V_h^{k},\label{12281757}
\end{align}
where $u_h^0=R_hu^0,v_h^0=R_hv^0$ and 
\begin{align}
&u^{\frac{1}{2}}=u^{0}+\frac{\tau}{2}u_t(\boldsymbol{x},0),~ u_t(\boldsymbol{x},0)=(\nu_1 + \mathrm{i}\alpha_1)\Delta u^0 -\left((\kappa_1 + \mathrm{i}\beta_1)f_1(\vert u^0 \vert^2)-(\mu_1 + \mathrm{i}\delta_1)g_1(\vert v^0 \vert^2)\right) u^0 + \gamma_1 u^0,\notag\\
&v^{\frac{1}{2}}=v^{0}+\frac{\tau}{2}v_t(\boldsymbol{x},0),~ v_t(\boldsymbol{x},0)=(\nu_2 + \mathrm{i}\alpha_2)\Delta v^0 - \left((\kappa_2 + \mathrm{i}\beta_2)f_2(\vert u^0 \vert^2)-(\mu_2 + \mathrm{i}\delta_2)g_2(\vert v^0 \vert^2)\right) v^0 + \gamma_2 v^0. \notag 
\end{align}
It is evident that when the time step \(\tau\) is sufficiently small, there exists a unique solution to the fully discrete numerical scheme \eqref{12281754}-\eqref{12281757}.
\subsection{Some important results}
In this subsection, we present several lemmas that are essential to the subsequent theoretical proofs. It should be emphasized that the transformation formula in Lemma \ref{lemma2.2} plays a pivotal role in establishing the boundedness of the numerical solution, which is proposed herein for the first time.
\begin{lemma}\label{lemma1}
Let $V$ be a complex inner product space, where \((\cdot,\cdot)\) denotes its inner product and \(\|\cdot\|\) is the norm induced by this inner product. Then, for any collection of vectors \(v^0, v^1, \cdots, v^N\) in $V$, the following statement holds:
\begin{align}
\text{Re}(D_\tau v^n,\hat{v}^n)=\frac{1}{4\tau}(G^n-G^{n-1})+\frac{\theta(1-\theta^2)}{8\tau}\|v^{n}-2v^{n-1}+v^{n-2}\|,\quad 2\leq n\leq N,\notag 
\end{align}
where
\begin{align}
G^n = (1+\theta)\|v^n\|^2+(1-\theta)\|v^{n-1}\|^2\geq \|v^n\|^2,\quad 1\leq n\leq N.\notag 
\end{align}
\end{lemma}
\begin{proof}
The result pertaining to real inner product spaces has been proven in Lemma 3.2 of \cite{Tan2025}. By following an analogous proof procedure, the result also holds for the complex case.
\end{proof}
\begin{lemma}\label{lemma2.2}
Let $V$ be a normed linear space equipped with the norm $\|\cdot\|$, and let $v^0, v^1, \cdots, v^N$ be elements of $V$. Then the following inequality holds for all integers $n$ satisfying $1\leq n\leq N$:
\begin{align}
\|v^n\|\leq 2\sum\limits_{k=1}^{n}\|\hat{v}^k\|+\|v^0\|.\label{12282110}
\end{align}
\begin{proof}
The proof of the assertion in \eqref{12282110} is carried out by means of mathematical induction. When \(n=1\), by virtue of the triangle inequality, we have
\begin{align}
\|v^1\|=\|2\hat{v}^1-v^0\|\leq 2\|\hat{v}^1\|+\|v^0\|.\label{12282126}
\end{align}
When $n=2$, using the definition of \(\hat{v}\) and \eqref{12282126}, we obtain
\begin{align}
\|v^2\|&=\left\|\frac{4}{2-\theta^2+\theta}\hat{v}^2-\frac{2\theta^2}{2-\theta^2+\theta}v^1-\frac{2-\theta^2-\theta}{2-\theta^2+\theta}v^0\right\|\notag\\
&\leq \frac{4}{2-\theta^2+\theta}\|\hat{v}^2\|+\frac{2\theta^2}{2-\theta^2+\theta}\|v^1\|+\frac{2-\theta^2-\theta}{2-\theta^2+\theta}\|v^0\|\notag\\
&\leq \frac{4}{2-\theta^2+\theta}\|\hat{v}^2\|+\frac{2\theta^2}{2-\theta^2+\theta}\left(2\|\hat{v}^1\|+\|v^0\|\right)+\frac{2-\theta^2-\theta}{2-\theta^2+\theta}\|v^0\|\notag\\
&\leq \frac{4}{2-\theta^2+\theta}\|\hat{v}^2\|+\frac{4\theta^2}{2-\theta^2+\theta}\|\hat{v}^1\|+\frac{2+\theta^2-\theta}{2-\theta^2+\theta}\|v^0\|\notag\\
&\leq 2\sum\limits_{k=1}^{2}\|\hat{v}^k\|+\|v^0\|,\label{12282143}
\end{align}
where we also use the following elementary inequalities:
\begin{align}
\frac{4}{2-\theta^2+\theta}\leq2,\quad \frac{4\theta^2}{2-\theta^2+\theta}\leq 2,\quad \frac{2+\theta^2-\theta}{2-\theta^2+\theta}\leq 1.\notag
\end{align}
Next, we assume that the inequality holds for \(1 \leq n\leq m\) $(m\geq2)$, i.e.,
\begin{align}
	\|v^n\|\leq 2\sum\limits_{k=1}^{n}\|\hat{v}^k\|+\|v^0\|.\label{202512282200}
\end{align}
The validity of the conclusion for \(n = m + 1\) is now established. In fact, following the similar procedure as in \eqref{12282143} and employing the inductive hypothesis \eqref{202512282200}, it holds that 
\begin{align*}
\|v^{m+1}\|&=\left\|\frac{4}{2-\theta^2+\theta}\hat{v}^{m+1}-\frac{2\theta^2}{2-\theta^2+\theta}v^{m}-\frac{2-\theta^2-\theta}{2-\theta^2+\theta}v^{m-1}\right\|\notag\\
&\leq \frac{4}{2-\theta^2+\theta}\|\hat{v}^{m+1}\|+\frac{2\theta^2}{2-\theta^2+\theta}\left(2\sum\limits_{k=1}^{m}\|\hat{v}^k\|+\|v^0\|\right)+\frac{2-\theta^2-\theta}{2-\theta^2+\theta}\left(2\sum\limits_{k=1}^{m-1}\|\hat{v}^k\|+\|v^0\|\right)\notag\\
&\leq \frac{4}{2-\theta^2+\theta}\|\hat{v}^{m+1}\|+\frac{4\theta^2}{2-\theta^2+\theta}\|\hat{v}^{m}\|+2\left(\frac{2+\theta^2-\theta}{2-\theta^2+\theta}\right)\sum\limits_{k=1}^{m-1}\|\hat{v}^k\|+\frac{2+\theta^2-\theta}{2-\theta^2+\theta}\|v^0\|\notag\\
&\leq 2\sum\limits_{k=1}^{m+1}\|\hat{v}^k\|+\|v^0\|.
\end{align*}
By the principle of mathematical induction, the conclusion holds. The proof is complete.
\end{proof}
\end{lemma}
\begin{lemma}\label{lemma3}
\text{(\cite{HeywoodRannacher1990})} Let \(a \geq 0\), \(b > 0\), and let \(\{ \eta_i \}_{i=0}^N\) and \(\{ \xi_i \}_{i=0}^N\) be two sequences of non-negative real numbers. Suppose that for all \(0 \leq n \leq N\), the following inequality holds:
\begin{align*}
\eta_n + \tau \sum_{i=0}^n \xi_i \leq a + b\tau \sum_{i=0}^n \eta_i.
\end{align*}
Then, if \(\tau \leq \frac{1}{2b}\), the stronger bound $$\eta_n + \tau \sum_{i=0}^n \xi_i \leq a \exp(2b(n+1)\tau)$$ is valid for all \(0 \leq n \leq N\).
\end{lemma}
\section{Theoretical analysis of the fully discrete numerical scheme}\label{section3}
This section is devoted to the theoretical analysis of the proposed numerical algorithm, covering its stability, the infinity-norm boundedness of the numerical solutions, and unconditionally optimal error estimates in the $L^2$ and $H^1$-norms.
\subsection{Stability of the fully discrete DLN scheme}
\begin{theorem}
Let $u_h^n$ and $v_h^n$ be the solutions of the fully discrete numerical algorithm  \eqref{12281754}-\eqref{12281757}, when \(\tau\) is sufficiently small, we have
\begin{align}
\|u_h^n\|+\|v_h^n\|\leq C(\|u_h^0\|+\|v_h^0\|), \quad 1\leq n\leq N.\notag 
\end{align}
\end{theorem}
\begin{proof}
First, we prove that the conclusion holds for \(n=1\). In fact, by taking \(\phi_h=\hat{u}_h^1\) and \(\psi_h=\hat{v}_h^1\) in \eqref{12281756} and \eqref{12281757}, respectively, taking the real part on both sides of the equations, and noting the following basic facts:
\begin{align}
&\text{Re}(D_\tau u_h^1,\hat{u}_h^1)=\frac{1}{2\tau}(\|u_h^1\|^2-\|u_h^0\|^2),\quad \text{Re}(D_\tau v_h^1,\hat{v}^1_h)=\frac{1}{2\tau}(\|v_h^1\|^2-\|v_h^0\|^2),\notag\\
&\text{Re}\left[(\kappa_1 + \mathrm{i}\beta_1)(f_1(\vert u^\frac{1}{2} \vert^2)\hat{u}_h^1,\hat{u}_h^1)+(\mu_1 + \mathrm{i}\delta_1)(g_1(\vert v^\frac{1}{2} \vert^2)\hat{u}_h^1,\hat{u}_h^1)\right]\geq0,\notag\\
&\text{Re}\left[(\kappa_2 + \mathrm{i}\beta_2)(f_2(\vert u^\frac{1}{2} \vert^2)\hat{v}_h^1,\hat{v}_h^1)+(\mu_2 + \mathrm{i}\delta_2)(g_2(\vert v^\frac{1}{2} \vert^2)\hat{v}_h^1,\hat{v}_h^1)\right]\geq0,\notag\\
&\text{Re}\left[(\nu_1 + \mathrm{i}\alpha_1)(\nabla \hat{u}_h^1,\nabla\hat{u}_h^1)\right]\geq0,\quad \text{Re}\left[(\nu_2 + \mathrm{i}\alpha_2)(\nabla \hat{v}_h^1,\nabla\hat{v}_h^1)\right]\geq0,\notag
\end{align}
we can immediately obtain 
\begin{align}
&\frac{1}{2\tau}(\|u_h^1\|^2-\|u_h^0\|^2)\leq |\gamma_1|\|\hat{u}_h^1\|^2\leq |\gamma_1|\left(\frac{\|u_h^1\|+\|u_h^0\|}{2}\right)^2, \notag\\
&\frac{1}{2\tau}(\|v_h^1\|^2-\|v_h^0\|^2)\leq |\gamma_2|\|\hat{v}_h^1\|^2\leq |\gamma_2|\left(\frac{\|v_h^1\|+\|v_h^0\|}{2}\right)^2.\notag
\end{align}
Canceling $\frac{\|u_h^1\|+\|u_h^0\|}{2}$ and $\frac{\|v_h^1\|+\|v_h^0\|}{2}$ from both sides of the corresponding equations, respectively, we get 
\begin{align}
\|u_h^1\|\leq \frac{1+\frac{|\gamma_1|}{2}\tau}{1-\frac{|\gamma_1|}{2}\tau}\|u_h^0\|\leq C\|u_h^0\|,\quad \|v_h^1\|\leq \frac{1+\frac{|\gamma_2|}{2}\tau}{1-\frac{|\gamma_2|}{2}\tau}\|v_h^0\|\leq C\|v_h^0\|.\label{202601072243}
\end{align}
Next, we prove the general case. By choosing the test functions $\phi_h=\hat{u}_h^n,\psi_h=\hat{v}_h^n$ in \eqref{12281754} and \eqref{12281755} respectively,  we have 
\begin{align}
	&(D_\tau u_h^n,\hat{u}_h^n)+(\nu_1 + \mathrm{i}\alpha_1)(\nabla \hat{u}_h^n,\nabla\hat{u}_h^n) + (\kappa_1 + \mathrm{i}\beta_1)(f_1(\vert \tilde{u}^n_h \vert^2)\hat{u}_h^n,\hat{u}_h^n)\notag\\
	&\quad + (\mu_1 + \mathrm{i}\delta_1)(g_1(\vert \tilde{v}^n_h \vert^2)\hat{u}_h^n,\hat{u}_h^n) - \gamma_1 \|\hat{u}_h^n\|^2 = 0,\quad 2\leq n\leq N,\label{20251229950}\\
	&(D_\tau v_h^n,\hat{v}_h^n) + (\nu_2 + \mathrm{i}\alpha_2)(\nabla \hat{v}_h^n,\nabla\hat{v}_h^n) + (\kappa_2 + \mathrm{i}\beta_2)(f_2(\vert \tilde{u}_h^n \vert^2)\hat{v}_h^n,\hat{v}_h^n)\notag\\
	&\quad + (\mu_2 + \mathrm{i}\delta_2)(g_2(\vert \tilde{v}_h^n  \vert^2)\hat{v}_h^n,\hat{v}_h^n) - \gamma_2 \|\hat{v}_h^n\|^2 = 0,\quad 2\leq n\leq N.\label{20251229951}
\end{align}
Define the following notations: 
\begin{align*}
F^n= (1+\theta)\|u^n_h\|^2+(1-\theta)\|u^{n-1}_h\|^2,\quad G^n= (1+\theta)\|v^n_h\|^2+(1-\theta)\|v^{n-1}_h\|^2,\quad 1\leq n\leq N.
\end{align*}
Taking the real part of both sides of \eqref{20251229950} and \eqref{20251229951}, respectively, it follows that
\begin{align}
&\frac{1}{4\tau}(F^n-F^{n-1})\leq |\gamma_1|\|\hat{u}_h^n\|^2\leq C(\|u_h^n\|^2+\|u_h^{n-1}\|+\|u_h^{n-2}\|^2),\label{202512291022}\\
&\frac{1}{4\tau}(G^n-G^{n-1})\leq |\gamma_2|\|\hat{v}_h^n\|^2\leq C(\|v_h^n\|^2+\|v_h^{n-1}\|+\|v_h^{n-2}\|^2),\label{202512291021}
\end{align}
where we also make use of the following basic fact given in Lemma \ref{lemma1}:
\begin{align*}
\text{Re}(D_\tau u_h^n,\hat{u}_h^n)\geq \frac{1}{4\tau}(F^n-F^{n-1}),\quad \text{Re}(D_\tau v_h^n,\hat{v}_h^n)\geq \frac{1}{4\tau}(G^n-G^{n-1}),
\end{align*}
and the non-negativity of the other relevant terms, i.e.,
\begin{align}
&\text{Re}\left[(\kappa_1 + \mathrm{i}\beta_1)(f_1(\vert \tilde{u}^n_h \vert^2)\hat{u}_h^n,\hat{u}_h^n)+(\mu_1 + \mathrm{i}\delta_1)(g_1(\vert \tilde{v}^n_h \vert^2)\hat{u}_h^n,\hat{u}_h^n)\right]\geq0,\notag\\
&\text{Re}\left[(\kappa_2 + \mathrm{i}\beta_2)(f_2(\vert \tilde{u}^n_h \vert^2)\hat{v}_h^n,\hat{v}_h^n)+(\mu_2 + \mathrm{i}\delta_2)(g_2(\vert \tilde{v}^n_h \vert^2)\hat{v}_h^n,\hat{v}_h^n)\right]\geq0,\notag\\
&\text{Re}\left[(\nu_1 + \mathrm{i}\alpha_1)(\nabla \hat{u}_h^n,\nabla\hat{u}_h^n)\right]\geq0,\quad \text{Re}\left[(\nu_2 + \mathrm{i}\alpha_2)(\nabla \hat{v}_h^n,\nabla\hat{v}_h^n)\right]\geq0.\notag
\end{align}
By summing the two sides of \eqref{202512291022} and \eqref{202512291021} separately and noting Lemma \ref{lemma1}, we derive
\begin{align}
&\|u_h^n\|^2\leq F^n\leq F^1+C\tau\sum\limits_{k=0}^n\|u_h^k\|^2\leq C(\|u_h^1\|^2+\|u_h^0\|^2)+C\tau\sum\limits_{k=0}^n\|u_h^k\|^2, \quad 2\leq n\leq N,\notag\\
&\|v_h^n\|^2\leq G^n\leq G^1+C\tau\sum\limits_{k=0}^n\|v_h^k\|^2\leq C(\|v_h^1\|^2+\|v_h^0\|^2)+C\tau\sum\limits_{k=0}^n\|v_h^k\|^2,\quad 2\leq n\leq N.\notag 
\end{align}
Obviously, the above two inequalities also hold for $n=0$ and $1$, and thus by Gronwall's inequality presented in Lemma \ref{lemma3}, we have
\begin{align}
\|u_h^n\|\leq C(\|u_h^1\|+\|u_h^0\|)\leq C\|u_h^0\|,\quad \|v_h^n\|\leq C(\|v_h^1\|+\|v_h^0\|)\leq C\|v_h^0\|,\quad 2\leq n\leq N.\label{202512291039}
\end{align}
Combining \eqref{202601072243} and \eqref{202512291039}, we conclude that the theorem holds true.
\end{proof}
\subsection{Boundedness of the numerical solutions in $L^{\infty}$-norm}
In this subsection, the $L^{\infty}$-boundedness of the numerical solutions is established via mathematical induction. Next, prior to establishing the estimate for the infinity norm, we first present two essential lemmas, namely Lemma \ref{lemma4} and Lemma \ref{lemma3.2}. For this purpose and subsequent convergence analysis, we introduce the following notations:
\begin{align*}
u(t_n)-u_h^n = u(t_n)-R_hu(t_n)+R_hu(t_n)-u_h^n=\xi_u^n+\eta_u^n, \quad 0 \leq n\leq N, \\
v(t_n)-v_h^n = v(t_n)-R_hv(t_n)+R_hv(t_n)-v_h^n=\xi_v^n+\eta_v^n, \quad 0 \leq n\leq N. 
\end{align*}
\begin{lemma}\label{lemma4}
Let \(u_h^1\) and \(v_h^1\) be the solutions corresponding to the fully discrete numerical algorithm in \eqref{12281756}-\eqref{12281757}, when the time step \(\tau\) is taken to be sufficiently small, we obtain
\begin{align}
\|\nabla \eta_u^1\|+\|\nabla \eta_v^1\|+\tau\left(\|\Delta_h \eta_u^1\|+\|\Delta_h \eta_v^1\|\right)\leq C(\tau^2+h^{k+1}).
\end{align}
\end{lemma}
\begin{proof}
From the variational formulation \eqref{12281756}-\eqref{12281757} of the equations, it follows that the exact solutions satisfy
\begin{align}
	&(D_\tau u^1,\phi_h)+(\nu_1 + \mathrm{i}\alpha_1)(\nabla \hat{u}^1,\nabla\phi_h) + (\kappa_1 + \mathrm{i}\beta_1)(f_1(\vert u^\frac{1}{2} \vert^2)\hat{u}^1,\phi_h)\notag\\
	&\quad + (\mu_1 + \mathrm{i}\delta_1)(g_1(\vert v^\frac{1}{2} \vert^2)\hat{u}^1,\phi_h) - \gamma_1 (\hat{u}^1,\phi_h) = (R^1_u,\phi_h), \quad \forall \phi_h \in V_h^{k},\label{202512292054}\\
	&(D_\tau v^1,\psi_h) + (\nu_2 + \mathrm{i}\alpha_2)(\nabla \hat{v}^1,\nabla\psi_h) + (\kappa_2 + \mathrm{i}\beta_2)(f_2(\vert u^\frac{1}{2} \vert^2)\hat{v}^1,\psi_h)\notag\\
	&\quad + (\mu_2 + \mathrm{i}\delta_2)(g_2(\vert v^\frac{1}{2}  \vert^2)\hat{v}^1,\psi_h) - \gamma_2 (\hat{v}^1,\psi_h) = (R^1_v,\psi_h),\quad \forall \psi_h\in V_h^{k},\label{202512292055}
\end{align}
where $R^1_u$ and $R^1_v$ are defined below 
\begin{align}
R^1_u &= \left(D_\tau u^1-u(t_{\frac{1}{2}})\right)-(\nu_1 + \mathrm{i}\alpha_1)\left(\Delta \hat{u}^1-\Delta u(t_{\frac{1}{2}})\right)+(\kappa_1 + \mathrm{i}\beta_1)\left(f_1(\vert u^\frac{1}{2} \vert^2)\hat{u}^1-f_1(\vert u(t_{\frac{1}{2}}) \vert^2)u(t_{\frac{1}{2}})\right)\notag\\
&\quad+(\mu_1 + \mathrm{i}\delta_1)\left(g_1(\vert v^\frac{1}{2} \vert^2)\hat{u}^1-g_1(\vert v(t_{\frac{1}{2}}) \vert^2)u(t_{\frac{1}{2}})\right)-\gamma_1 \left(\hat{u}^1-u(t_{\frac{1}{2}})\right),\notag\\
R^1_v &= \left(D_\tau v^1-v(t_{\frac{1}{2}})\right)-(\nu_2 + \mathrm{i}\alpha_2)\left(\Delta \hat{v}^1-\Delta v(t_{\frac{1}{2}})\right)+(\kappa_2 + \mathrm{i}\beta_2)\left(f_2(\vert u^\frac{1}{2} \vert^2)\hat{v}^1-f_2(\vert u(t_{\frac{1}{2}}) \vert^2)v(t_{\frac{1}{2}})\right)\notag\\
&\quad+(\mu_2 + \mathrm{i}\delta_2)\left(g_2(\vert v^\frac{1}{2} \vert^2)\hat{v}^1-g_2(\vert v(t_{\frac{1}{2}}) \vert^2)v(t_{\frac{1}{2}})\right)-\gamma_2 \left(\hat{v}^1-v(t_{\frac{1}{2}})\right)\notag.
\end{align}
By subtracting the equations \eqref{12281756}-\eqref{12281757} from \eqref{202512292054}-\eqref{202512292055} respectively, and invoking the definition of the Ritz projection operator, we arrive at the following results:
\begin{align}
&(D_\tau \eta_u^1,\phi_h)+(\nu_1 + \mathrm{i}\alpha_1)(\nabla \hat{\eta}_u^1,\nabla\phi_h) =- (\kappa_1 + \mathrm{i}\beta_1)\left(f_1(\vert u^\frac{1}{2} \vert^2)\hat{u}^1-f_1(\vert u^\frac{1}{2} \vert^2)\hat{u}_h^1,\phi_h\right)\notag\\
&\quad - (\mu_1 + \mathrm{i}\delta_1)\left(g_1(\vert v^\frac{1}{2} \vert^2)\hat{u}^1-g_1(\vert v^\frac{1}{2} \vert^2)\hat{u}_h^1,\phi_h\right) + \gamma_1 (\hat{\eta}_u^1,\phi_h)\notag\\
&\quad + (R^1_u,\phi_h)-(D_\tau \xi_u^1,\phi_h)+\gamma_1 (\hat{\xi}_u^1,\phi_h), \quad \forall \phi_h \in V_h^{k},\label{202512292126}\\
&(D_\tau \eta_v^1,\psi_h) + (\nu_2 + \mathrm{i}\alpha_2)(\nabla \hat{\eta}_v^1,\nabla\psi_h) =- (\kappa_2 + \mathrm{i}\beta_2)\left(f_2(\vert u^\frac{1}{2} \vert^2)\hat{v}^1-f_2(\vert u^\frac{1}{2} \vert^2)\hat{v}_h^1,\psi_h\right)\notag\\
&\quad - (\mu_2 + \mathrm{i}\delta_2)\left(g_2(\vert v^\frac{1}{2}  \vert^2)\hat{v}^1-g_2(\vert v^\frac{1}{2}  \vert^2)\hat{v}_h^1,\psi_h\right) + \gamma_2 (\hat{\eta}_v^1,\psi_h)\notag\\
&\quad + (R^1_v,\psi_h)-(D_\tau \xi_v^1,\psi_h)+\gamma_2 (\hat{\xi}_v^1,\psi_h),\quad \forall \psi_h\in V_h^{k}.\label{202512292127}
\end{align}
Choosing $\phi_h=-\Delta_h\hat{\eta}_u^1$ and $\psi_h=-\Delta_h\hat{\eta}_v^1$ in \eqref{202512292126} and \eqref{202512292127}, respectively, taking the real part of each side of the corresponding equation and making use of the following basic facts: 
\begin{align}
&\text{Re}(D_\tau \eta_u^1,-\Delta_h\hat{\eta}_u^1) =  \frac{1}{2\tau}\left(\|\nabla\eta_u^1\|^2-\|\nabla\eta_u^0\|^2\right),\quad \text{Re}(D_\tau \eta_v^1,-\Delta_h\hat{\eta}_v^1) =  \frac{1}{2\tau}\left(\|\nabla\eta_v^1\|^2-\|\nabla\eta_v^0\|^2\right),\notag\\
&(\nabla \hat{\eta}_u^1,\nabla(-\Delta_h\hat{\eta}_u^1))=\|\Delta_h\hat{\eta}_u^1\|^2,\quad (\nabla \hat{\eta}_v^1,\nabla(-\Delta_h\hat{\eta}_v^1))=\|\Delta_h\hat{\eta}_v^1\|^2,\notag
\end{align}
it holds that 
\begin{align}
&\|\nabla\eta_u^1\|^2-\|\nabla\eta_u^0\|^2+2\tau\nu_1\|\Delta_h\hat{\eta}_u^1\|^2 =-2\tau\text{Re}\left[(\kappa_1 + \mathrm{i}\beta_1)\left(f_1(\vert u^\frac{1}{2} \vert^2)\hat{u}^1-f_1(\vert u^\frac{1}{2} \vert^2)\hat{u}_h^1,-\Delta_h\hat{\eta}_u^1\right)\right]\notag\\
&\quad -2\tau\text{Re}\left[ (\mu_1 + \mathrm{i}\delta_1)\left(g_1(\vert v^\frac{1}{2} \vert^2)\hat{u}^1-g_1(\vert v^\frac{1}{2} \vert^2)\hat{u}_h^1,-\Delta_h\hat{\eta}_u^1\right)\right] + 2\tau\gamma_1  \text{Re}(\hat{\eta}_u^1,-\Delta_h\hat{\eta}_u^1)\notag\\
&\quad + 2\tau\text{Re}(R^1_u,-\Delta_h\hat{\eta}_u^1)-2\tau\text{Re}(D_\tau \xi_u^1,-\Delta_h\hat{\eta}_u^1)+2\tau\gamma_1\text{Re} (\hat{\xi}_u^1,-\Delta_h\hat{\eta}_u^1)\notag\\
\quad&\equiv A_1+A_2+A_3+A_4+A_5+A_6,\label{202512292329}\\
&\|\nabla\eta_v^1\|^2-\|\nabla\eta_v^0\|^2 + 2\tau\nu_2 \|\Delta_h\hat{\eta}_v^1\|^2 =-2\tau\text{Re}\left[ (\kappa_2 + \mathrm{i}\beta_2)\left(f_2(\vert u^\frac{1}{2} \vert^2)\hat{v}^1-f_2(\vert u^\frac{1}{2} \vert^2)\hat{v}_h^1,-\Delta_h\hat{\eta}_v^1\right)\right]\notag\\
&\quad -2\tau\text{Re}\left[ (\mu_2 + \mathrm{i}\delta_2)\left(g_2(\vert v^\frac{1}{2} \vert^2)\hat{v}^1-g_2(\vert v^\frac{1}{2}  \vert^2)\hat{v}_h^1,-\Delta_h\hat{\eta}_v^1\right)\right]+ 2\tau\gamma_2 \text{Re} (\hat{\eta}_v^1,-\Delta_h\hat{\eta}_v^1)\notag\\	
&\quad +2\tau\text{Re} (R^1_v,-\Delta_h\hat{\eta}_v^1)-2\tau\text{Re}(D_\tau \xi_v^1,-\Delta_h\hat{\eta}_v^1)+2\tau\gamma_2 \text{Re}(\hat{\xi}_v^1,-\Delta_h\hat{\eta}_v^1)\notag\\
\quad&\equiv A_7+A_8+A_9+A_{10}+A_{11}+A_{12}.\label{202512292330}
\end{align}
To complete the proof, it suffices to derive an upper bound for each term on the right-hand side of the aforementioned equations. We first focus our attention on the term \(A_1\). Specifically, by virtue of the Cauchy-Schwarz inequality and Young's inequality, we establish that
\begin{align}
A_1&= -2\tau\text{Re}\left[(\kappa_1 + \mathrm{i}\beta_1)\left(f_1(\vert u^\frac{1}{2} \vert^2)\hat{u}^1-f_1(\vert u^\frac{1}{2} \vert^2)\hat{u}_h^1,-\Delta_h\hat{\eta}_u^1\right)\right]\notag\\
&\leq \frac{1}{6}\tau\nu_1\|\Delta_h\hat{\eta}_u^1\|^2+C\tau\|\hat{u}^1-\hat{u}_h^1\|^2\notag\\
&\leq \frac{1}{6}\tau\nu_1\|\Delta_h\hat{\eta}_u^1\|^2+C\tau\|\hat{\eta}_u^1\|^2+Ch^{2k+2}\notag\\
&\leq \frac{1}{6}\tau\nu_1\|\Delta_h\hat{\eta}_u^1\|^2+C\tau\|\nabla\hat{\eta}_u^1\|^2+Ch^{2k+2}.\label{202512301716}
\end{align}
The estimate for \(A_2\) can be obtained by the same reasoning as \(A_1\), and we omit the detailed proof here for brevity
\begin{align}
A_2& = -2\tau\text{Re}\left[ (\mu_1 + \mathrm{i}\delta_1)\left(g_1(\vert v^\frac{1}{2} \vert^2)\hat{u}^1-g_1(\vert v^\frac{1}{2} \vert^2)\hat{u}_h^1,-\Delta_h\hat{\eta}_u^1\right)\right]\notag\\ 
&\leq \frac{1}{6}\tau\nu_1\|\Delta_h\hat{\eta}_u^1\|^2+C\tau\|\nabla\hat{\eta}_u^1\|^2+Ch^{2k+2}.\label{202512301717}
\end{align}
By a further application of Young's inequality and the Cauchy-Schwarz inequality,
the estimates for $A_3, A_4, A_5$ and $A_6$ follow readily
\begin{align}
A_3+A_4+A_5+A_6&=2\tau\gamma_1  \text{Re}(\hat{\eta}_u^1,-\Delta_h\hat{\eta}_u^1)
 + 2\tau\text{Re}(R^1_u,-\Delta_h\hat{\eta}_u^1)\notag\\
 &\quad -2\tau\text{Re}(D_\tau \xi_u^1,-\Delta_h\hat{\eta}_u^1)+2\tau\gamma_1\text{Re} (\hat{\xi}_u^1,-\Delta_h\hat{\eta}_u^1)\notag\\
 &\leq \frac{2}{3}\tau\nu_1\|\Delta_h\hat{\eta}_u^1\|^2 + C\tau\|\nabla\hat{\eta}_u^1\|^2+C\tau\|R^1_u\|^2+C\tau\|D_\tau \xi_u^1\|^2+C\tau\|\hat{\xi}_u^1\|^2\notag\\
 &\leq \frac{2}{3}\tau\nu_1\|\Delta_h\hat{\eta}_u^1\|^2 + C\tau\|\nabla\hat{\eta}_u^1\|^2+C\tau^4+Ch^{2k+2}.\label{202512301718}
\end{align}
Using similar techniques, we obtain the estimates for \(A_7\) to \(A_{12}\)
\begin{align}
A_7+A_8+A_9+A_{10}+A_{11}+A_{12}\leq \tau\nu_2\|\Delta_h\hat{\eta}_v^1\|^2+C\tau\|\nabla\hat{\eta}_v^1\|^2+C\tau^4+Ch^{2k+2}.\label{202512301721}
\end{align}
By substituting the results of \eqref{202512301716}-\eqref{202512301718} and \eqref{202512301721} into \eqref{202512292329} and \eqref{202512292330}, respectively, we get
\begin{align}
(1-C\tau)\|\nabla\eta_u^1\|^2+\tau\nu_1\|\Delta_h\eta_u^1\|^2\leq C\tau^4+h^{2k+2},\notag\\
(1-C\tau)\|\nabla\eta_v^1\|^2+\tau\nu_1\|\Delta_h\eta_v^1\|^2\leq C\tau^4+h^{2k+2}.\notag 
\end{align}
When \(\tau\) is sufficiently small, the conclusion of the lemma holds. This completes the proof.
\end{proof}
\begin{lemma}\label{lemma3.2}
Let \(u_h^n\) and \(v_h^n\) be the solutions of the fully discrete numerical algorithm \eqref{12281754}-\eqref{12281757}. If the time step \(\tau\) is chosen to be sufficiently small, we have
\begin{align}
\|\nabla \eta_u^n\|+\|\nabla \eta_v^n\|+\tau\left(\|\Delta_h \eta_u^n\|+\|\Delta_h \eta_v^n\|\right)\leq C(\tau^2+h^{k+1}),\quad 0\leq n\leq N.\notag
\end{align}
\end{lemma}
\begin{proof}
To establish the conclusion, we employ the method of mathematical induction in the following proof. When \(n = 0\), the result is clearly satisfied (as it holds trivially). The scenario corresponding to \(n = 1\) has been proven in Lemma \ref{lemma4}. Suppose that the statement holds true for all \( n \leq m-1~(m\geq 2)\), i.e., 
\begin{align}
\|\nabla \eta_u^n\|+\|\nabla \eta_v^n\|+\tau\left(\|\Delta_h \eta_u^n\|+\|\Delta_h \eta_v^n\|\right)\leq C(\tau^2+h^{k+1}),\quad 0\leq n\leq m-1.\label{202512302050}
\end{align}
Next, we prove the boundedness of $\eta_u^n$ and $\eta_v^n$ in the infinity norm by considering two separate cases. When \(\tau\leq h\), we obtain the desired result by means of the inverse inequality \eqref{202512302043}, \eqref{202512302050} and Poincaré's inequality. In fact, it arrives at 
\begin{align}
\|\eta_u^n\|_{\infty} + \|\eta_v^n\|_{\infty}&\leq Ch^{-\frac{d}{2}}(\|\eta_u^n\| + \|\eta_v^n\|)\notag\\
&\leq Ch^{-\frac{d}{2}}(\|\nabla \eta_u^n\| + \|\nabla \eta_v^n\|)\notag\\
&\leq Ch^{2-\frac{d}{2}}\leq 1,\quad 0\leq n\leq m-1. \notag 
\end{align}
For the case \(\tau > h\), by applying the discrete Agmon inequality \eqref{202512302053} and \eqref{202512302050}, we can obtain the corresponding result, i.e., 
\begin{align}
\|\eta_u^n\|_\infty+\|\eta_v^n\|_\infty&\leq C\|\nabla \eta_u^n \|^{1/2} \| \Delta_h \eta_u^n \|^{1/2}+C\|\nabla \eta_v^n \|^{1/2} \| \Delta_h \eta_v^n \|^{1/2}\notag\\
&\leq C\tau^{\frac{3}{2}}\leq 1,\quad 0\leq n\leq m-1. \notag
\end{align}
Combining the above two cases, we conclude that for \(0\leq n\leq m-1\), the following holds:
\begin{align}
\|\eta_u^n\|_\infty+\|\eta_v^n\|_\infty\leq 1, \notag
\end{align}
which further gives that
\begin{align}
\|u_h^n\|_\infty+\|v_h^n\|_\infty\leq C, \quad 0\leq n\leq m-1.\label{202512311008}
\end{align}
In what follows, we focus on proving that the conclusion of the lemma holds for $n=m$. For this purpose, simple algebraic manipulations yield the equations satisfied by the exact solutions, which are listed as follows:
\begin{align}
&(D_\tau u^n,\phi_h)+(\nu_1 + \mathrm{i}\alpha_1)(\nabla \hat{u}^n,\nabla\phi_h) + (\kappa_1 + \mathrm{i}\beta_1)(f_1(\vert \tilde{u}^n \vert^2)\hat{u}^n,\phi_h)\notag\\
&\quad + (\mu_1 + \mathrm{i}\delta_1)(g_1(\vert \tilde{v}^n \vert^2)\hat{u}^n,\phi_h) - \gamma_1 (\hat{u}^n,\phi_h) = (R^n_u,\phi_h), \quad \forall \phi_h \in V_h^{k},\quad2\leq n\leq m,\label{202512302126}\\
&(D_\tau v^n,\psi_h) + (\nu_2 + \mathrm{i}\alpha_2)(\nabla \hat{v}^n,\nabla\psi_h) + (\kappa_2 + \mathrm{i}\beta_2)(f_2(\vert \tilde{u}^n \vert^2)\hat{v}^n,\psi_h)\notag\\	&\quad + (\mu_2 + \mathrm{i}\delta_2)(g_2(\vert \tilde{v}^n  \vert^2)\hat{v}^n,\psi_h) - \gamma_2 (\hat{v}^n,\psi_h) = (R^n_v,\psi_h),\quad \forall \psi_h\in V_h^{k},\quad 2\leq n\leq m,\label{202512302127}
\end{align}
where $R^n_u$ and $R^n_v$ are defined subsequently
\begin{align}
	R^n_u &= \left(D_\tau u^n-u(t_{n-1+\frac{\theta}{2}})\right)-(\nu_1 + \mathrm{i}\alpha_1)\left(\Delta \hat{u}^n-\Delta u(t_{n-1+\frac{\theta}{2}})\right)\notag\\
	&\quad +(\kappa_1 + \mathrm{i}\beta_1)\left(f_1(\vert \tilde{u}^n \vert^2)\hat{u}^n-f_1(\vert u(t_{n-1+\frac{\theta}{2}}) \vert^2)u(t_{n-1+\frac{\theta}{2}})\right)\notag\\
	&\quad+(\mu_1 + \mathrm{i}\delta_1)\left(g_1(\vert \tilde{v}^n \vert^2)\hat{u}^n-g_1(\vert v(t_{n-1+\frac{\theta}{2}}) \vert^2)u(t_{n-1+\frac{\theta}{2}})\right)-\gamma_1 \left(\hat{u}^n-u(t_{n-1+\frac{\theta}{2}})\right),\notag\\
	R^n_v &= \left(D_\tau v^n-v(t_{n-1+\frac{\theta}{2}})\right)-(\nu_2 + \mathrm{i}\alpha_2)\left(\Delta \hat{v}^n-\Delta v(t_{n-1+\frac{\theta}{2}})\right)\notag\\
	&\quad+(\kappa_2 + \mathrm{i}\beta_2)\left(f_2(\vert \tilde{u}^n \vert^2)\hat{v}^n-f_2(\vert u(t_{n-1+\frac{\theta}{2}}) \vert^2)v(t_{n-1+\frac{\theta}{2}})\right)\notag\\
	&\quad+(\mu_2 + \mathrm{i}\delta_2)\left(g_2(\vert \tilde{v}^n \vert^2)\hat{v}^n-g_2(\vert v(t_{n-1+\frac{\theta}{2}}) \vert^2)v(t_{n-1+\frac{\theta}{2}})\right)-\gamma_2 \left(\hat{v}^n-v(t_{n-1+\frac{\theta}{2}})\right)\notag.
\end{align}
By combining \eqref{12281754}-\eqref{12281755} and \eqref{202512302126}-\eqref{202512302127}, we have
\begin{align}
&(D_\tau \eta_u^n,\phi_h)+(\nu_1 + \mathrm{i}\alpha_1)(\nabla \hat{\eta}_u^n,\nabla\phi_h) =- (\kappa_1 + \mathrm{i}\beta_1)\left(f_1(\vert \tilde{u}^n \vert^2)\hat{u}^n-f_1(\vert \tilde{u}_h^n \vert^2)\hat{u}_h^n,\phi_h\right)\notag\\
&\quad - (\mu_1 + \mathrm{i}\delta_1)\left(g_1(|\tilde{v}^n \vert^2)\hat{u}^n-g_1(\vert \tilde{v}_h^n \vert^2)\hat{u}_h^n,\phi_h\right) + \gamma_1 (\hat{\eta}_u^n,\phi_h)\notag\\
&\quad + (R^n_u,\phi_h)-(D_\tau \xi_u^n,\phi_h)+\gamma_1 (\hat{\xi}_u^n,\phi_h), \quad \forall \phi_h \in V_h^{k},\label{202512302244}\\
&(D_\tau \eta_v^n,\psi_h) + (\nu_2 + \mathrm{i}\alpha_2)(\nabla \hat{\eta}_v^n,\nabla\psi_h) =- (\kappa_2 + \mathrm{i}\beta_2)\left(f_2(\vert \tilde{u}^n \vert^2)\hat{v}^n-f_2(\vert \tilde{u}_h^n \vert^2)\hat{v}_h^n,\psi_h\right)\notag\\
&\quad - (\mu_2 + \mathrm{i}\delta_2)\left(g_2(\vert \tilde{v}^n  \vert^2)\hat{v}^n-g_2(\vert \tilde{v}^n_h  \vert^2)\hat{v}_h^n,\psi_h\right) + \gamma_2 (\hat{\eta}_v^n,\psi_h)\notag\\
&\quad + (R^n_v,\psi_h)-(D_\tau \xi_v^n,\psi_h)+\gamma_2 (\hat{\xi}_v^n,\psi_h),\quad \forall \psi_h\in V_h^{k}.\label{202512302245}
\end{align}
In \eqref{202512302244} and \eqref{202512302245}, set \(\phi_h = -\Delta_h\hat{\eta}_u^n\) and \(\psi_h = -\Delta_h\hat{\eta}_v^n\) (for each equation, respectively), then take the real part of both sides of each resulting equation, it holds that 
\begin{align}
	&F_u^n-F_u^{n-1}+4\tau\nu_1\|\Delta_h\hat{\eta}_u^n\|^2 \leq -4\tau\text{Re}\left[(\kappa_1 + \mathrm{i}\beta_1)\left(f_1(\vert \tilde{u}^n \vert^2)\hat{u}^n-f_1(\vert \tilde{u}_h^n \vert^2)\hat{u}_h^n,-\Delta_h\hat{\eta}_u^n\right)\right]\notag\\
	&\quad -4\tau\text{Re}\left[ (\mu_1 + \mathrm{i}\delta_1)\left(g_1(\vert \tilde{v}^n \vert^2)\hat{u}^n-g_1(\vert \tilde{v}_h^n \vert^2)\hat{u}_h^n,-\Delta_h\hat{\eta}_u^n\right)\right] + 4\tau\gamma_1  \text{Re}(\hat{\eta}_u^n,-\Delta_h\hat{\eta}_u^n)\notag\\
	&\quad + 4\tau\text{Re}(R^n_u,-\Delta_h\hat{\eta}_u^n)-4\tau\text{Re}(D_\tau \xi_u^n,-\Delta_h\hat{\eta}_u^n)+4\tau\gamma_1\text{Re} (\hat{\xi}_u^n,-\Delta_h\hat{\eta}_u^n)\notag\\
	\quad&\equiv A_u^1+A_u^2+A_u^3+A_u^4+A_u^5+A_u^6,\label{20251231924}\\
	&F_v^n-F_v^{n-1} + 4\tau\nu_2 \|\Delta_h\hat{\eta}_v^n\|^2 \leq -4\tau\text{Re}\left[ (\kappa_2 + \mathrm{i}\beta_2)\left(f_2(\vert \tilde{u}^n \vert^2)\hat{v}^n-f_2(\vert \tilde{u}^n_h \vert^2)\hat{v}_h^n,-\Delta_h\hat{\eta}_v^n\right)\right]\notag\\
	&\quad -4\tau\text{Re}\left[ (\mu_2 + \mathrm{i}\delta_2)\left(g_2(\vert \tilde{v}^n \vert^2)\hat{v}^n-g_2(\vert \tilde{v}^n_h  \vert^2)\hat{v}_h^n,-\Delta_h\hat{\eta}_v^n\right)\right]+ 4\tau\gamma_2 \text{Re} (\hat{\eta}_v^n,-\Delta_h\hat{\eta}_v^n)\notag\\	
	&\quad +4\tau\text{Re} (R^n_v,-\Delta_h\hat{\eta}_v^n)-4\tau\text{Re}(D_\tau \xi_v^n,-\Delta_h\hat{\eta}_v^n)+4\tau\gamma_2 \text{Re}(\hat{\xi}_v^n,-\Delta_h\hat{\eta}_v^n)\notag\\
	\quad&\equiv A_v^1+A_v^2+A_v^3+A_v^4+A_v^5+A_v^6,\label{202512311043}
\end{align}
here, we make use of the results given in Lemma \ref{lemma1}, which are presented as follows: 
\begin{align}
&(D_\tau \eta_u^n,-\Delta_h\hat{\eta}_u^n)\geq \frac{1}{4\tau}(F_u^n-F_u^{n-1}), \quad F_u^n= (1+\theta)\|\nabla \eta_u^n\|^2+(1-\theta)\|\nabla \eta_u^{n-1}\|^2\geq \|\nabla \eta_u^n\|^2,\notag\\
&(D_\tau \eta_v^n,-\Delta_h\hat{\eta}_v^n)\geq \frac{1}{4\tau}(F_v^n-F_v^{n-1}), \quad F_v^n= (1+\theta)\|\nabla \eta_v^n\|^2+(1-\theta)\|\nabla \eta_v^{n-1}\|^2\geq \|\nabla \eta_v^n\|^2.\notag
\end{align}
We now provide estimates for each term in the inequalities above. For $A_u^1$, employing the Young's inequality and Cauchy-Schwarz inequality, we arrive at 
\begin{align}
A_u^1&=-4\tau\text{Re}\left[(\kappa_1 + \mathrm{i}\beta_1)\left(f_1(\vert \tilde{u}^n \vert^2)\hat{u}^n-f_1(\vert \tilde{u}_h^n \vert^2)\hat{u}_h^n,-\Delta_h\hat{\eta}_u^n\right)\right]\notag\\
&\leq \frac{1}{2}\tau\nu_1\|\Delta_h\hat{\eta}_u^1\|^2+C\tau\|f_1(\vert \tilde{u}^n \vert^2)\hat{u}^n-f_1(\vert \tilde{u}_h^n \vert^2)\hat{u}_h^n\|^2\notag\\
&\leq \frac{1}{2}\tau\nu_1\|\Delta_h\hat{\eta}_u^1\|^2+C\tau\|f_1(\vert \tilde{u}^n \vert^2)(\hat{u}^n-\hat{u}_h^n)+\left(f_1(\vert \tilde{u}^n \vert^2)-f_1(\vert \tilde{u}^n_h \vert^2)\right)\hat{u}_h^n\|^2\notag\\
&\leq \frac{1}{2}\tau\nu_1\|\Delta_h\hat{\eta}_u^n\|^2+C\tau\|\hat{u}^n-\hat{u}_h^n\|^2+C\tau\|\left(\vert \tilde{u}^n \vert^2-\vert \tilde{u}^n_h \vert^2\right)\hat{u}_h^n\|^2\notag\\
&\leq \frac{1}{2}\tau\nu_1\|\Delta_h\hat{\eta}_u^n\|^2+C\tau\|\hat{\eta}_u^n\|^2+C\tau h^{2k+2}+C\tau\|\tilde{u}^n-\tilde{u}_h^n\|^2\notag\\
&\leq \frac{1}{2}\tau\nu_1\|\Delta_h\hat{\eta}_u^n\|^2+C\tau\|\hat{\eta}_u^n\|^2+C\tau h^{2k+2}+C\tau\|\tilde{\eta}_u^n\|^2\notag\\
&\leq \frac{1}{2}\tau\nu_1\|\Delta_h\hat{\eta}_u^n\|^2+C\tau\|\nabla\hat{\eta}_u^n\|^2+C\tau h^{2k+2}+C\tau\|\nabla\tilde{\eta}_u^n\|^2,\label{202512311042}
\end{align}
where we have used \eqref{202512311008} and the following inequalities in the argument:
\begin{align}
	&\|f_1(\vert \tilde{u}^n \vert^2)\|_{\infty}\leq C,\quad |f_1(\vert \tilde{u}^n \vert^2)-f_1(\vert \tilde{u}^n_h \vert^2)|\leq C(\vert \tilde{u}^n \vert^2-\vert \tilde{u}^n_h \vert^2),\notag\\
	&||a|^2-|b|^2|\leq |a+b||a-b|,\quad \forall a,b\in\mathbb{ C}, \quad \hat{u}_h^n=R_hu(t_n)-\eta_{u}^n.\notag
\end{align}
Following the above analysis, we obtain the estimate for \(A_u^2\)
\begin{align}
A_u^2 &= -4\tau\text{Re}\left[ (\mu_1 + \mathrm{i}\delta_1)\left(g_1(\vert \tilde{v}^n \vert^2)\hat{u}^n-g_1(\vert \tilde{v}_h^n \vert^2)\hat{u}_h^n,-\Delta_h\hat{\eta}_u^n\right)\right]\notag\\
& \leq \frac{1}{2}\tau\nu_1\|\Delta_h\hat{\eta}_u^n\|^2+C\tau\|\nabla\hat{\eta}_u^n\|^2+C\tau h^{2k+2}+C\tau\|\nabla\tilde{\eta}_v^n\|^2\label{202512311040}
\end{align}
For the remaining terms, the estimates can be obtained easily 
\begin{align}
A_u^3+A_u^4+A_u^5+A_u^6 &= 4\tau\gamma_1  \text{Re}(\hat{\eta}_u^n,-\Delta_h\hat{\eta}_u^n)+ 4\tau\text{Re}(R^n_u,-\Delta_h\hat{\eta}_u^n)\notag\\
&\quad -4\tau\text{Re}(D_\tau \xi_u^n,-\Delta_h\hat{\eta}_u^n)+4\tau\gamma_1\text{Re} (\hat{\xi}_u^n,-\Delta_h\hat{\eta}_u^n)\notag\\
&\leq 2\tau\nu_1\|\Delta_h\hat{\eta}_u^n\|^2+C\tau\|\nabla\hat{\eta}_u^n\|^2+C\tau^5+C\tau h^{2k+2}.\label{202512311039}
\end{align}
By virtue of the estimates for $u$ derived above, we can establish the corresponding estimates for $v$ as follows:
\begin{align}
A_v^1+A_v^2+A_v^3+A_v^4+A_v^5+A_v^6&\leq 3\tau\nu_2\|\Delta_h\hat{\eta}_v^n\|^2+C\tau\|\nabla\hat{\eta}_v^n\|^2+C\tau\|\nabla\tilde{\eta}_v^n\|^2\notag\\
&\quad+C\tau\|\nabla\tilde{\eta}_u^n\|^2+C\tau^5+C\tau h^{2k+2}.\label{202512311038}
\end{align}
Substituting \eqref{202512311042}-\eqref{202512311039} into \eqref{20251231924} and \eqref{202512311038} into \eqref{202512311043} respectively, followed by summing the resulting two inequalities, yields
\begin{align}
&\quad(F_u^n+F_v^{n})-(F_u^{n-1}+F_v^{n-1})+\tau\left(\nu_1\|\Delta_h\hat{\eta}_u^n\|^2+\nu_2\|\Delta_h\hat{\eta}_v^n\|^2\right)\notag\\
&\leq C\tau\|\nabla\hat{\eta}_v^n\|^2+C\tau\|\nabla\hat{\eta}_u^n\|^2+C\tau\|\nabla\tilde{\eta}_u^n\|^2+C\tau\|\nabla\tilde{\eta}_v^n\|^2+C\tau^5+C\tau h^{2k+2},\quad 2\leq n\leq m.\notag 
\end{align}
Summing both sides of the aforementioned inequality and invoking Lemma \ref{lemma4}, we arrive at
\begin{align}
&\quad (\|\nabla\eta_u^n\|^2+\|\nabla\eta_v^n\|^2)+\tau\sum\limits_{k=2}^{n}\left(\nu_1\|\Delta_h\hat{\eta}_u^k\|^2+\nu_2\|\Delta_h\hat{\eta}_v^k\|^2\right)\notag\\
&\leq(F_u^n+F_v^{n})+\tau\sum\limits_{k=2}^{n}\left(\nu_1\|\Delta_h\hat{\eta}_u^k\|^2+\nu_2\|\Delta_h\hat{\eta}_v^k\|^2\right)\notag\\
&\leq (F_u^{1}+F_v^{1}) +\tau \sum\limits_{k=2}^{n}(\|\nabla\eta_u^k\|^2+\|\nabla\eta_v^k\|^2)+C\tau^4+Ch^{2k+2}\notag\\
&\leq C\tau \sum\limits_{k=2}^{n}(\|\nabla\eta_u^k\|^2+\|\nabla\eta_v^k\|^2)+C\tau^4+Ch^{2k+2}, \quad 2\leq n\leq m.\notag
\end{align}
By the discrete Gromwall inequality presented in Lemma \ref{lemma3} and the Cauchy-Schwarz inequality in $\mathbb{R}^n$, we have  
\begin{align}
&\quad\left(\|\nabla\eta_u^m\|^2+\|\nabla\eta_v^m\|^2\right)+\tau\sum\limits_{k=2}^{m}\left(\|\Delta_h\hat{\eta}_u^k\|^2+\|\Delta_h\hat{\eta}_v^k\|^2\right)\leq C\tau^4+Ch^{2k+2}. \notag 
\end{align}
From the transfer formula in Lemma \ref{lemma2.2}, we get 
\begin{align}
\tau\left(\|\Delta_h\eta_u^m\|+\|\Delta_h\eta_v^m\|\right)
&\leq 2\tau\left(\sum\limits_{k=1}^{m}\|\Delta_h\hat{\eta}_u^k\|+\sum\limits_{k=1}^{m}\|\Delta_h\hat{\eta}_v^k\|\right)\notag\\
&\leq 2\sqrt{T}\left(\sqrt{\tau\sum\limits_{k=1}^{m}\|\Delta_h\hat{\eta}_u^k\|^2}+\sqrt{\tau\sum\limits_{k=1}^{m}\|\Delta_h\hat{\eta}_v^k\|^2}\right)\notag\\
&\leq C(\tau^2+h^{k+1}).\notag 
\end{align}
By the principle of mathematical induction, we conclude that the conclusion of the lemma holds true.
\end{proof}
By virtue of the result established in Lemma \ref{lemma3.2}, we can derive the boundedness of the numerical solutions with respect to the infinity norm.
\begin{theorem}\label{thm1}
The solutions to the fully discrete numerical scheme \eqref{12281754}-\eqref{12281757} are bounded in the infinity norm, that is, there exists a positive constant $C$ independent of the mesh parameters $\tau$ and $h$ such that
\begin{align}
\|u_h^n\|_\infty +\|v_h^n\|_\infty \leq C,\quad 0\leq n\leq N.\notag 
\end{align}
\begin{proof}
The proof of this theorem is already implicit in that of Lemma \ref{lemma3.2}. For the sake of reader convenience, we provide a complete and self-contained proof herein. The argument is divided into two distinct cases for the sake of clarity.\\
$\mathbf{Case~I}: \tau \leq h$\\
With the help of the discrete inverse inequality, it holds that 
\begin{align}
\|u_h^n\|_{\infty} +\|v_h^n\|_{\infty}&\leq \|R_hu^n\|_{\infty}+\|R_hv^n\|_{\infty}+\|\eta_u^n\|_{\infty}+\|\eta_v^n\|_{\infty}\notag\\
&\leq C+ Ch^{2-\frac{d}{2}}\leq C,\quad 0\leq n\leq N.\notag 
\end{align}
$\mathbf{Case~II}: \tau > h$\\
By virtue of the discrete Agmon inequality, we can derive
\begin{align}
\|u_h^n\|_{\infty} +\|v_h^n\|_{\infty}&\leq \|R_hu^n\|_{\infty}+\|R_hv^n\|_{\infty}+\|\eta_u^n\|_{\infty}+\|\eta_v^n\|_{\infty}\notag\\ 
&\leq C+C\| \nabla \eta_u^n \|^{1/2} \| \Delta_h \eta_u^n \|^{1/2}+C\| \nabla \eta_v^n \|^{1/2} \| \Delta_h \eta_v^n \|^{1/2}\notag\\
&\leq C+C\tau^{\frac{3}{2}}\leq C,\quad 0\leq n\leq N.\notag 
\end{align}
All this completes the proof.
\end{proof}
\end{theorem}
By exploiting the boundedness of the numerical solutions in $L^\infty$-norm, we can  establish the unconditionally optimal error estimates for the fully discrete scheme. 
\begin{theorem}\label{thm2}
Let $u^n$ and $v^n$ be the solutions to the continuous equations \eqref{202512272133}-\eqref{202512272136}, and let $u^n_h$ and $v^n_h$ be the solutions to the fully discrete numerical scheme \eqref{12281754}-\eqref{12281757}. For sufficiently small $\tau$ and $h$, we have
\begin{align}
&\|u^n-u_h^n\|+\|v^n-v_h^n\| \leq C(\tau^2+h^{k+1}),\quad 0\leq n\leq N,\notag\\
&\|\nabla u^n-\nabla u_h^n\|+\|\nabla v^n-\nabla v_h^n\|\leq C(\tau^2+h^{k}),\quad 0\leq n\leq N.\notag
\end{align}
\begin{proof}
We first prove the optimal convergence rate in the energy norm. For this purpose, using the triangle inequality, the approximation property of the Ritz projection operator, and Lemma \ref{lemma3.2}, we have 
\begin{align*}
\|\nabla u^n-\nabla u_h^n\|+\|\nabla v^n-\nabla v_h^n\|\leq \|\nabla \eta_u^n\|+\|\nabla \eta_v^n\|+\|\nabla \xi_u^n\|+\|\nabla \xi_v^n\|\leq C(\tau^2+h^{k}).
\end{align*}
Subsequently, we focus on deriving the error estimate in the $L^2$ norm. To this end, substituting $\phi_h=\hat{\eta}_u^1$ and $\psi_h=\hat{\eta}_v^1$ into \eqref{202512292126} and \eqref{202512292127}, respectively, and then taking the real part of both sides of each equation subsequently gives
\begin{align}
&\|\eta_u^1\|^2-\|\eta_u^0\|^2+2\tau\nu_1\|\nabla\hat{\eta}_u^1\|^2 =-2\tau\text{Re}\left[(\kappa_1 + \mathrm{i}\beta_1)\left(f_1(\vert u^\frac{1}{2} \vert^2)\hat{u}^1-f_1(\vert u^\frac{1}{2} \vert^2)\hat{u}_h^1,\hat{\eta}_u^1\right)\right]\notag\\
&\quad -2\tau\text{Re}\left[ (\mu_1 + \mathrm{i}\delta_1)\left(g_1(\vert v^\frac{1}{2} \vert^2)\hat{u}^1-g_1(\vert v^\frac{1}{2} \vert^2)\hat{u}_h^1,\hat{\eta}_u^1\right)\right] + 2\tau\gamma_1  \text{Re}(\hat{\eta}_u^1,\hat{\eta}_u^1)\notag\\
&\quad + 2\tau\text{Re}(R^1_u,\hat{\eta}_u^1)-2\tau\text{Re}(D_\tau \xi_u^1,\hat{\eta}_u^1)+2\tau\gamma_1\text{Re} (\hat{\xi}_u^1,\hat{\eta}_u^1)\notag\\
\quad&\equiv B_1+B_2+B_3+B_4+B_5+B_6,\notag\\
&\|\eta_v^1\|^2-\|\eta_v^0\|^2 + 2\tau\nu_2 \|\nabla\hat{\eta}_v^1\|^2 =-2\tau\text{Re}\left[ (\kappa_2 + \mathrm{i}\beta_2)\left(f_2(\vert u^\frac{1}{2} \vert^2)\hat{v}^1-f_2(\vert u^\frac{1}{2} \vert^2)\hat{v}_h^1,\hat{\eta}_v^1\right)\right]\notag\\
&\quad -2\tau\text{Re}\left[ (\mu_2 + \mathrm{i}\delta_2)\left(g_2(\vert v^\frac{1}{2} \vert^2)\hat{v}^1-g_2(\vert v^\frac{1}{2}  \vert^2)\hat{v}_h^1,\hat{\eta}_v^1\right)\right]+ 2\tau\gamma_2 \text{Re} (\hat{\eta}_v^1,\hat{\eta}_v^1)\notag\\	
&\quad +2\tau\text{Re} (R^1_v,\hat{\eta}_v^1)-2\tau\text{Re}(D_\tau \xi_v^1,\hat{\eta}_v^1)+2\tau\gamma_2 \text{Re}(\hat{\xi}_v^1,\hat{\eta}_v^1)\notag\\
\quad&\equiv B_7+B_8+B_9+B_{10}+B_{11}+B_{12}.\notag 
\end{align}
By following the approach similar to the one employed in \eqref{202512301716}-\eqref{202512301721}, we obtain 
\begin{align}
&B_1+B_2+B_3+B_4+B_5+B_6\leq 2\tau\nu_1\|\nabla\hat{\eta}_u^1\|^2+C\tau\|\hat{\eta}_u^1\|^2+C\tau^4+Ch^{2k+2},\label{202512311916}\notag\\
&B_7+B_8+B_9+B_{10}+B_{11}+B_{12}\leq 2\tau\nu_2\|\nabla\hat{\eta}_v^1\|^2+C\tau\|\hat{\eta}_v^1\|^2+C\tau^4+Ch^{2k+2}.\notag
\end{align}
From the combination of these two results, we have
\begin{align}
\|\eta_u^1\|+\|\eta_v^1\|\leq C(\tau^2+h^{k+1}).
\end{align}
Substituting \(\phi_h = \hat{\eta}_u^n\) and \(\psi_h = \hat{\eta}_v^n\) into \eqref{202512302244} and \eqref{202512302245}, respectively, and extracting the real part from both sides of the equations then yields
\begin{align}
	&G_u^n-G_u^{n-1}+4\tau\nu_1\|\nabla\hat{\eta}_u^n\|^2 \leq -4\tau\text{Re}\left[(\kappa_1 + \mathrm{i}\beta_1)\left(f_1(\vert \tilde{u}^n \vert^2)\hat{u}^n-f_1(\vert \tilde{u}_h^n \vert^2)\hat{u}_h^n,\hat{\eta}_u^n\right)\right]\notag\\
	&\quad -4\tau\text{Re}\left[ (\mu_1 + \mathrm{i}\delta_1)\left(g_1(\vert \tilde{v}^n \vert^2)\hat{u}^n-g_1(\vert \tilde{v}_h^n \vert^2)\hat{u}_h^n,\hat{\eta}_u^n\right)\right] + 4\tau\gamma_1  \text{Re}(\hat{\eta}_u^n,\hat{\eta}_u^n)\notag\\
	&\quad + 4\tau\text{Re}(R^n_u,\hat{\eta}_u^n)-4\tau\text{Re}(D_\tau \xi_u^n,\hat{\eta}_u^n)+4\tau\gamma_1\text{Re} (\hat{\xi}_u^n,\hat{\eta}_u^n)\notag\\
	\quad&\equiv B_u^1+B_u^2+B_u^3+B_u^4+B_u^5+B_u^6,\label{202512311956}\\
	&G_v^n-G_v^{n-1} + 4\tau\nu_2 \|\nabla\hat{\eta}_v^n\|^2 \leq -4\tau\text{Re}\left[ (\kappa_2 + \mathrm{i}\beta_2)\left(f_2(\vert \tilde{u}^n \vert^2)\hat{v}^n-f_2(\vert \tilde{u}^n_h \vert^2)\hat{v}_h^n,\hat{\eta}_v^n\right)\right]\notag\\
	&\quad -4\tau\text{Re}\left[ (\mu_2 + \mathrm{i}\delta_2)\left(g_2(\vert \tilde{v}^n \vert^2)\hat{v}^n-g_2(\vert \tilde{v}^n_h  \vert^2)\hat{v}_h^n,\hat{\eta}_v^n\right)\right]+ 4\tau\gamma_2 \text{Re} (\hat{\eta}_v^n,\hat{\eta}_v^n)\notag\\	
	&\quad +4\tau\text{Re} (R^n_v,\hat{\eta}_v^n)-4\tau\text{Re}(D_\tau \xi_v^n,\hat{\eta}_v^n)+4\tau\gamma_2 \text{Re}(\hat{\xi}_v^n,\hat{\eta}_v^n)\notag\\
	\quad&\equiv B_v^1+B_v^2+B_v^3+B_v^4+B_v^5+B_v^6,\label{202512311957}
\end{align}
where the following notations are adopted:
\begin{align}
	&(D_\tau \eta_u^n,\hat{\eta}_u^n)\geq \frac{1}{4\tau}(G_u^n-G_u^{n-1}), \quad G_u^n= (1+\theta)\|\eta_u^n\|^2+(1-\theta)\| \eta_u^{n-1}\|^2\geq \|\eta_u^n\|^2,\notag\\
	&(D_\tau \eta_v^n,\hat{\eta}_v^n)\geq \frac{1}{4\tau}(G_v^n-G_v^{n-1}), \quad G_v^n= (1+\theta)\| \eta_v^n\|^2+(1-\theta)\| \eta_v^{n-1}\|^2\geq \| \eta_v^n\|^2.\notag
\end{align}
Following the same approach as in \eqref{202512311042}-\eqref{202512311039} and exploiting boundedness of the numerical solutions in $L^{\infty}$-norm, we arrive at
\begin{align}
&\quad B_u^1+B_u^2+B_u^3+B_u^4+B_u^5+B_u^6\notag\\
&\leq 4\tau\nu_1\|\nabla\hat{\eta}_u^n\|^2+C\tau\|\hat{\eta}_u^n\|^2+C\tau\|\tilde{\eta}_v^n\|^2+C\tau\|\tilde{\eta}_u^n\|^2+C\tau^5+C\tau h^{2k+2},\label{202512312010}\\
&\quad B_v^1+B_v^2+B_v^3+B_v^4+B_v^5+B_v^6 \notag\\
&\leq 4\tau\nu_2\|\nabla_h\hat{\eta}_v^n\|^2+C\tau\|\hat{\eta}_v^n\|^2+C\tau\|\tilde{\eta}_u^n\|^2+C\tau\|\nabla\tilde{\eta}_v^n\|^2+C\tau^5+C\tau h^{2k+2}.\label{202512312016}
\end{align}
Substituting \eqref{202512312010}-\eqref{202512312016} into \eqref{202512311956}-\eqref{202512311957}, respectively, we then obtain
\begin{align}
&\quad(G_u^n+G_v^{n})-(G_u^{n-1}+G_v^{n-1})\notag\\
&\leq C\tau\|\hat{\eta}_u^n\|^2+ C\tau\|\hat{\eta}_v^n\|^2+C\tau\|\tilde{\eta}_u^n\|^2+C\tau\|\tilde{\eta}_v^n\|^2+C\tau^5+C\tau h^{2k+2},\quad 2\leq n\leq N.\notag 
\end{align}
Upon summing both sides of the above inequality, we arrive at
\begin{align*}
\|\eta_u^n\|^2+\|\eta_v^n\|^2&\leq G_u^{1}+G_v^{1} +\tau \sum\limits_{k=2}^{n}(\|\eta_u^k\|^2+\|\eta_v^k\|^2)+C\tau^4+Ch^{2k+2}\notag\\
&\leq C\tau \sum\limits_{k=2}^{n}(\|\eta_u^k\|^2+\|\eta_v^k\|^2)+C\tau^4+Ch^{2k+2}, \quad 2\leq n\leq N.
\end{align*}
With the help of the discrete Gronwall inequality, it holds that 
\begin{align*}
\|\eta_u^n\|+\|\eta_v^n\|\leq C(\tau^2+h^{k+1}),\quad 2\leq n\leq N.
\end{align*}
In the last, employing the triangle inequality yields 
\begin{align*}
\|u^n-u_{h}^n\|+\|v^n-v_{h}^n\|\leq \|\eta_{u}^n\|+\|\eta_{v}^n\|+\|\xi_{u}^n\|+\|\xi_{v}^n\|\leq C(\tau^2+h^{k+1}),\quad 0\leq n\leq N.
\end{align*}
This concludes the proof of the theorem.
\end{proof}
\end{theorem}
\section{Numerical examples}\label{section4}
In this section, three numerical examples, including both two- and three-dimensional cases, are presented to verify the correctness of the theoretical results. All numerical results in this paper are implemented with the aid of the open source finite element software FreeFEM \cite{Hecht2012}. In all examples, we test the convergence accuracy of the following errors for the numerical solutions $u_h^n$ and $v_h^n$:
\begin{align}
&E^1_{u} = \|\nabla(u^N-u_h^N)\|,\quad E^0_{u} = \|u^N-u_h^N\|,\notag\\
&E^1_{v} = \|\nabla(v^N-v_h^N)\|,\quad E^0_{v} = \|v^N-v_h^N\|.\notag 
\end{align}
\begin{example}\label{example1}
In the first  numerical example, we consider the following initial-boundary value problem for the coupled Ginzburg-Landau system: 
\begin{align}
&u_t - (1 + \mathrm{i})\Delta u + (1 + \mathrm{i})\vert u \vert^2u+(1 + \mathrm{i})\vert v \vert^2u - u = f_1(x,y,t), \quad (x,y,t) \in [0,1]^3, \label{202601042211}\\
&v_t - (1 + \mathrm{i})\Delta v + (1 + \mathrm{i})\vert u \vert^2v+(1 + \mathrm{i})\vert v \vert^2v - v = f_2(x,y,t), \quad (x,y,t) \in [0,1]^3, \label{202601042212}
\end{align}
with the exact solutions 
\begin{align}
&u(x,y,t) = i\exp(t)\sin(\pi x)\sin(\pi y),\quad v(x,y,t) = e^{it^2}\sin(x)(1-x)\sin(y)(1-y).\notag
\end{align}
The right-hand side terms \(f_1\) and \(f_2\) can be easily computed from the aforementioned exact  solutions.

The numerical results of this example are presented in Tables \ref{table1}--\ref{table5}. To test the convergence accuracy in the spatial direction, we fix a sufficiently small time step $\tau$ and select different spatial mesh sizes $h=1/5, 1/10, 1/15, 1/20, 1/25$. The numerical results corresponding to different polynomial degrees $k=1,2,3$ and different parameters $\theta = 0.25,0.5,0.75$ are presented in Tables \ref{table1}--\ref{table3}, respectively. It can be seen from the data in the tables that the convergence accuracy of the $L^2$ and $H^1$-norm errors in the spatial direction achieves the theoretically predicted orders. When testing the temporal accuracy, we set $\tau=h$. Tables \ref{table4} and \ref{table5} present the numerical results for $k=2,3$ and $\theta=0.1,0.9$, respectively. It can be observed from these numerical results that the temporal scheme achieves the theoretically predicted second-order convergence rate.
In addition, we also present the verification of the mesh-ratio independence in Figure \ref{figure1}. It can be seen from the figure that when a time step $\tau$ is fixed, the $L^2$-norm error tends to a constant as the spatial mesh is refined. The above numerical results fully demonstrate the robustness and effectiveness of the numerical algorithm constructed in this paper with respect to the parameters.

\begin{table}
	\centering 
	\caption{Convergence rates for Example \ref{example1} in the spatial direction with simulation parameters \(\theta=0.25\) and \(k=1\).}
	\label{table1}
	\begin{tabular}{ccccccccc}
		\toprule 
		$h$ & $E^1_u$ & $\text{Order}$ & $E^0_u$ & $\text{Order}$ & $E^1_v$ & $\text{Order}$ & $E^0_v$ & $\text{Order}$\\
		\midrule
$1/5$ & 6.8428e-01 &       --       & 3.9094e-02 &     --         & 4.3151e-02 &     --         & 2.4896e-03 &        --      \\
$1/10$ & 3.4723e-01 & 0.9787 & 9.9975e-03 & 1.9673 & 2.1863e-02 & 0.9809 & 6.3330e-04 & 1.9749 \\
$1/15$ & 2.3212e-01 & 0.9932 & 4.4620e-03 & 1.9897 & 1.4611e-02 & 0.9940 & 2.8237e-04 & 1.9921 \\
$1/20$ & 1.7426e-01 & 0.9967 & 2.5136e-03 & 1.9949 & 1.0967e-02 & 0.9970 & 1.5901e-04 & 1.9961 \\
$1/25$ & 1.3947e-01 & 0.9980 & 1.6098e-03 & 1.9970 & 8.7774e-03 & 0.9982 & 1.0182e-04 & 1.9977 \\
		\bottomrule
	\end{tabular}
\end{table}
\begin{table}
	\centering 
	\caption{Convergence rates for Example \ref{example1} in the spatial direction with simulation parameters \(\theta=0.5\) and \(k=2\).}
	\label{table2}
	\begin{tabular}{ccccccccc}
		\toprule 
		$h$ & $E^1_u$ & $\text{Order}$ & $E^0_u$ & $\text{Order}$ & $E^1_v$ & $\text{Order}$ & $E^0_v$ & $\text{Order}$\\
\midrule
$1/5$ & 8.5176e-02 &         --     & 1.9199e-03 &      --        & 4.9292e-03 &      --        & 1.0456e-04 &      --        \\
$1/10$ & 2.1535e-02 & 1.9838 & 2.4308e-04 & 2.9815 & 1.2415e-03 & 1.9892 & 1.3152e-05 & 2.9910 \\
$1/15$ & 9.5910e-03 & 1.9949 & 7.2194e-05 & 2.9942 & 5.5254e-04 & 1.9966 & 3.9013e-06 & 2.9972 \\
$1/20$ & 5.3989e-03 & 1.9975 & 3.0482e-05 & 2.9971 & 3.1095e-04 & 1.9983 & 1.6465e-06 & 2.9986\\
$1/25$ & 3.4565e-03 & 1.9985 & 1.5613e-05 & 2.9983 & 1.9905e-04 & 1.9990 & 8.4316e-07 & 2.9992 \\
\bottomrule
	\end{tabular}
\end{table}
\begin{table}
	\centering 
	\caption{Convergence rates for Example \ref{example1} in the spatial direction with simulation parameters \(\theta=0.75\) and \(k=3\).}
	\label{table3}
	\begin{tabular}{ccccccccc}
		\toprule 
$h$ & $E^1_u$ & $\text{Order}$ & $E^0_u$ & $\text{Order}$ & $E^1_v$ & $\text{Order}$ & $E^0_v$ & $\text{Order}$\\
		\midrule
$1/5$ & 6.0571e-03 &     --         & 1.7462e-04 &      --        & 2.6972e-04 &      --        & 7.2693e-06 &      --        \\
$1/10$ & 7.5441e-04 & 3.0052 & 1.1042e-05 & 3.9832 & 3.3270e-05 & 3.0192 & 4.5403e-07 & 4.0010 \\
$1/15$ & 2.1956e-04 & 3.0442 & 2.1822e-06 & 3.9988 & 9.6455e-06 & 3.0537 & 8.9461e-08 & 4.0061 \\
$1/20$ & 9.0346e-05 & 3.0867 & 6.8969e-07 & 4.0038 & 3.9544e-06 & 3.0995 & 2.8223e-08 & 4.0102\\
$1/25$ & 4.4922e-05 & 3.1313 & 2.8219e-07 & 4.0049 & 1.9582e-06 & 3.1495 & 1.1530e-08 & 4.0119 \\
		\bottomrule
	\end{tabular}
\end{table}

\begin{table}
	\centering 
	\caption{Convergence rates for Example \ref{example1} in the temporal direction with simulation parameters \(\theta=0.1\) and \(k=2\).}
	\label{table4}
	\begin{tabular}{ccccccccc}
		\toprule 
		$\tau$ & $E^1_u$ & $\text{Order}$ & $E^0_u$ & $\text{Order}$ & $E^1_v$ & $\text{Order}$ & $E^0_v$ & $\text{Order}$\\
		\midrule
	$1/5$ & 2.4432e-01 &        --      & 2.1696e-02 &       --       & 1.3196e-02 &      --        & 2.7083e-03 &       --       \\
	$1/10$ & 6.1992e-02 & 1.9786 & 4.8780e-03 & 2.1531 & 3.3263e-03 & 1.9882 & 6.8918e-04 & 1.9744 \\
	$1/15$ & 2.7765e-02 & 1.9810 & 2.2025e-03 & 1.9610 & 1.4392e-03 & 2.0663 & 2.9714e-04 & 2.0749 \\
	$1/20$ & 1.5643e-02 & 1.9944 & 1.2344e-03 & 2.0126 & 8.0934e-04 & 2.0008 & 1.6722e-04 & 1.9983 \\
	$1/25$ & 1.0018e-02 & 1.9969 & 7.8830e-04 & 2.0099 & 5.1742e-04 & 2.0048 & 1.0690e-04 & 2.0051 \\
		\bottomrule
	\end{tabular}
\end{table}

\begin{table}
	\centering 
	\caption{Convergence rates for Example \ref{example1} in the temporal direction with simulation parameters \(\theta=0.9\) and \(k=3\).}
	\label{table5}
	\begin{tabular}{ccccccccc}
		\toprule 
		$\tau$ & $E^1_u$ & $\text{Order}$ & $E^0_u$ & $\text{Order}$ & $E^1_v$ & $\text{Order}$ & $E^0_v$ & $\text{Order}$\\
		\midrule
		$1/5$ & 2.1736e-02 &       --       & 3.8386e-03 &       --       & 4.8911e-03 &        --      & 1.0955e-03 &       --       \\
		$1/10$ & 4.6175e-03 & 2.2349 & 9.3041e-04 & 2.0446 & 1.2432e-03 & 1.9760 & 2.7870e-04 & 1.9748 \\
		$1/15$ & 1.9626e-03 & 2.1101 & 4.0692e-04 & 2.0397 & 5.5608e-04 & 1.9843 & 1.2468e-04 & 1.9839 \\
		$1/20$ & 1.0843e-03 & 2.0624 & 2.2713e-04 & 2.0268 & 3.1389e-04 & 1.9879 & 7.0378e-05 & 1.9877 \\
		$1/25$ & 6.8760e-04 & 2.0413 & 1.4470e-04 & 2.0204 & 2.0134e-04 & 1.9900 & 4.5144e-05 & 1.9899 \\
		\bottomrule
	\end{tabular}
\end{table}
\end{example}

\begin{figure}[htb]
	\centering
	\subfigure[\scriptsize $L^2$-norm errors for $u$]{
		\label{$L^2$-norm errors for u} 
		\includegraphics[width=2.5in]{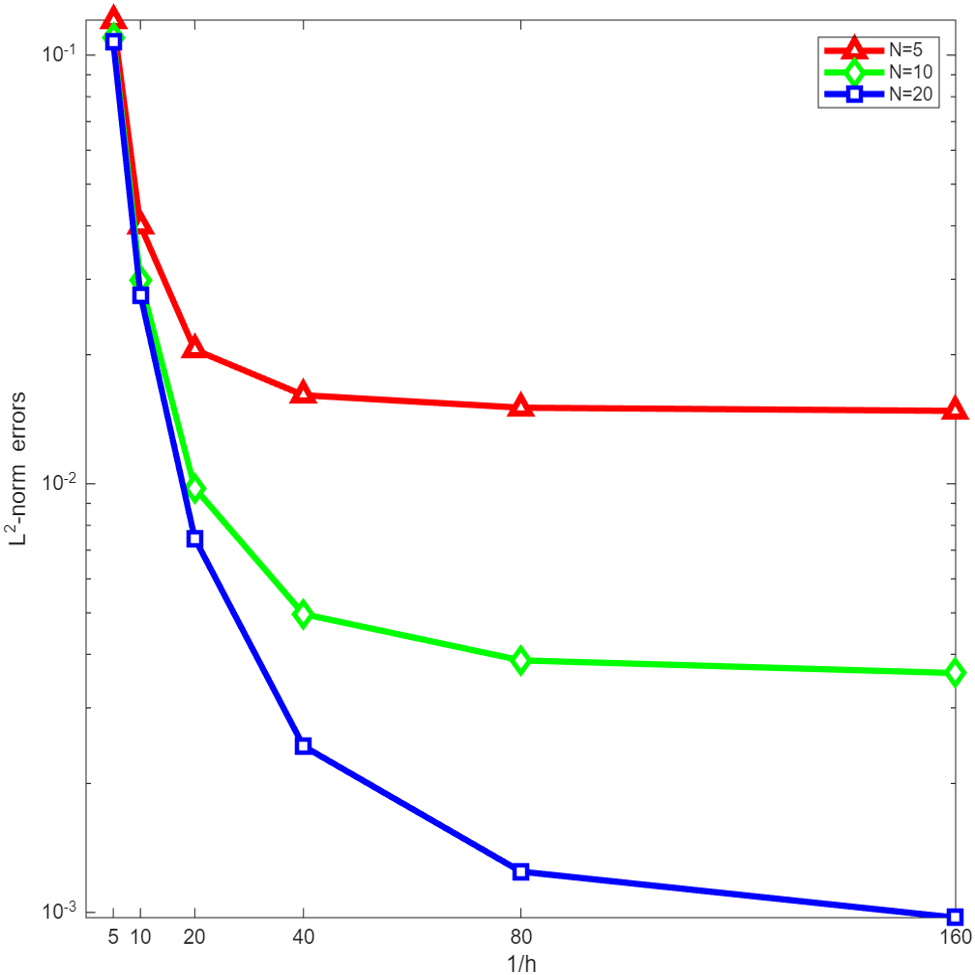}} 
	\hspace{20pt} 
	\subfigure[\scriptsize $L^2$-norm errors for $v$]{
		\label{$L^2$-norm errors for v} 
		\includegraphics[width=2.5in]{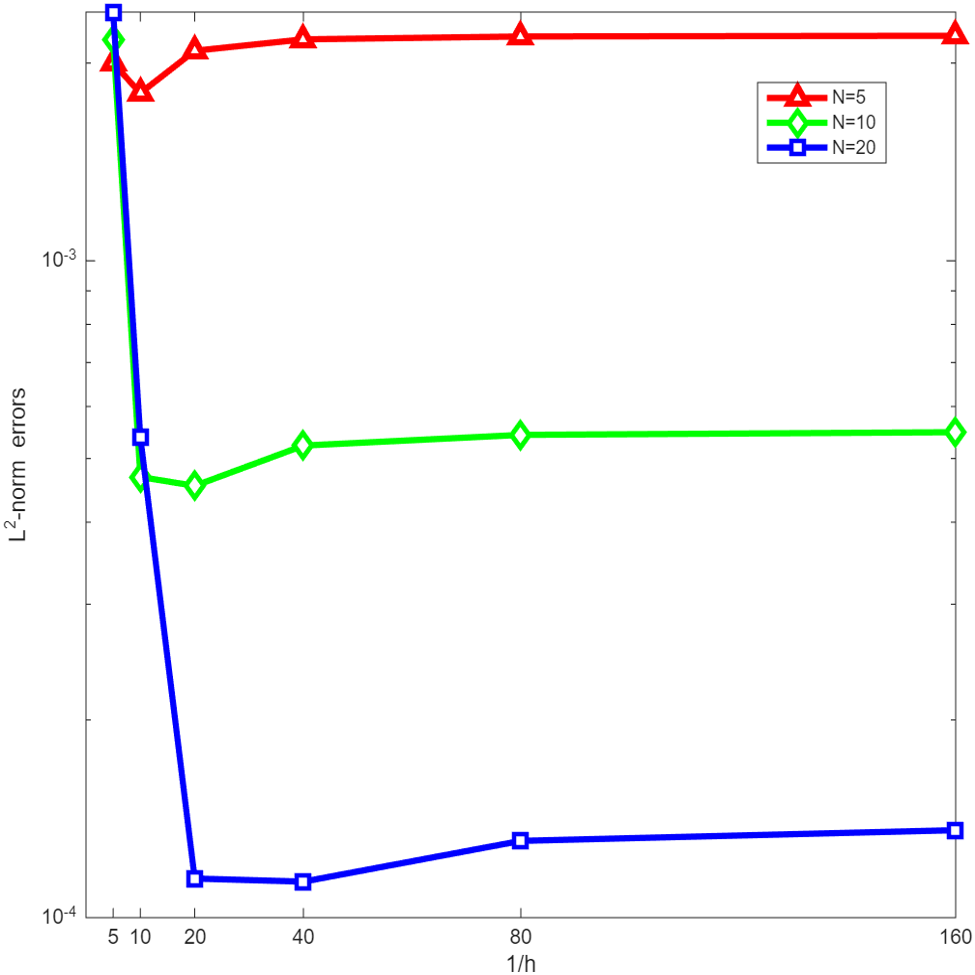}}
	\hspace{20pt}
	\caption{Trend of the $L^2$-norm errors of $u$ and $v$ as the mesh size $h$ decreases for Example \ref{example1} with parameters $k=1,\theta = 0.5$.}
	\label{figure1} 
\end{figure}

\begin{example} \label{example2}
In the present example, we consider and investigate a two-dimensional system of coupled Ginzburg-Landau complex equations that incorporate significant nonlinear effects, which is stated as follows:
\begin{align}
	&u_t - (1 + 2\mathrm{i})\Delta u + (1 + 3\mathrm{i})(\vert u \vert^2+\vert u \vert^4)u+(1 +4 \mathrm{i})(\vert v \vert^2+\vert v \vert^4)u - 5u = f_1(x,y,t), \notag\\
	&v_t - (5 + \mathrm{i})\Delta v + (4 + \mathrm{i})(\vert u \vert^2+\vert u \vert^4)v+(3 + \mathrm{i})(\vert v \vert^2+\vert v \vert^4)v - 2v = f_2(x,y,t), \notag
\end{align}
where $(x,y,t)\in[0,1]^3$. For the purpose of assessing the convergence rate, we adopt the following exact solutions:
\begin{align}
u(x,y,t) = e^{it}(1+5t^2)x(1-x)\sin(\pi y), \quad v(x,y,t) = (1+3i)(t+1)^2x(1-x)y(1-y).\notag 
\end{align}
and complex-valued functions \(f_1\) and \(f_2\) are obtained by substituting the two above exact solutions into the governing system. 

In the second example, we investigate the effectiveness of the numerical algorithm for high-degree nonlinear terms. Tables \ref{table6} to \ref{table10} present the errors and convergence orders in the temporal and spatial directions, and Figure \ref{figure2} shows the verification of the unconditional convergence of the numerical scheme. It further illustrates the applicability of the numerical scheme constructed in this paper to different nonlinear terms.
\begin{table}
	\centering 
	\caption{Convergence rates for Example \ref{example2} in the spatial direction with simulation parameters \(\theta=0.0\) and \(k=1\).}
	\label{table6}
	\begin{tabular}{ccccccccc}
		\toprule 
		$h$ & $E^1_u$ & $\text{Order}$ & $E^0_u$ & $\text{Order}$ & $E^1_v$ & $\text{Order}$ & $E^0_v$ & $\text{Order}$\\
		\midrule
		$1/5$ & 1.8103e-01 &       --       & 1.0407e-02 &     --         & 1.5127e-01 &     --         & 8.7464e-03 &       --       \\
		$1/10$ & 9.1783e-02 & 0.9799 & 2.6533e-03 & 1.9717 & 7.6644e-02 & 0.9809 & 2.2238e-03 & 1.9757 \\
		$1/15$ & 6.1347e-02 & 0.9936 & 1.1835e-03 & 1.9911 & 5.1220e-02 & 0.9940 & 9.9143e-04 & 1.9924 \\
		$1/20$ & 4.6052e-02 & 0.9969 & 6.6657e-04 & 1.9956 & 3.8448e-02 & 0.9970 & 5.5828e-04 & 1.9962 \\
		$1/25$ & 3.6857e-02 & 0.9981 & 4.2686e-04 & 1.9974 & 3.0771e-02 & 0.9982 & 3.5748e-04 & 1.9978\\
		\bottomrule
	\end{tabular}
\end{table}

\begin{table}
	\centering 
	\caption{Convergence rates for Example \ref{example2} in the spatial direction with simulation parameters \(\theta=0.5\) and \(k=2\).}
	\label{table7}
	\begin{tabular}{ccccccccc}
	\toprule 
$h$ & $E^1_u$ & $\text{Order}$ & $E^0_u$ & $\text{Order}$ & $E^1_v$ & $\text{Order}$ & $E^0_v$ & $\text{Order}$\\
\midrule
$1/5$ & 2.1463e-02 &     --         & 4.6907e-04 &       --       & 1.7032e-02 &       --       & 3.6021e-04 &        --      \\
$1/10$ & 5.4120e-03 & 1.9876 & 5.9215e-05 & 2.9858 & 4.2910e-03 & 1.9888 & 4.5250e-05 & 2.9928 \\
$1/15$ & 2.4092e-03 & 1.9961 & 1.7577e-05 & 2.9955 & 1.9099e-03 & 1.9965 & 1.3420e-05 & 2.9977 \\
$1/20$ & 1.3559e-03 & 1.9981 & 7.4199e-06 & 2.9978 & 1.0749e-03 & 1.9982 & 5.6635e-06 & 2.9987 \\
$1/25$ & 8.6800e-04 & 1.9989 & 3.8001e-06 & 2.9987 & 6.8809e-04 & 1.9988 & 2.9003e-06 & 2.9990 \\
		\bottomrule
	\end{tabular}
\end{table}

\begin{table}
	\centering 
	\caption{Convergence rates for Example \ref{example2} in the spatial direction with simulation parameters \(\theta=1.0\) and \(k=3\).}
	\label{table8}
	\begin{tabular}{ccccccccc}
		\toprule 
	    $h$ & $E^1_u$ & $\text{Order}$ & $E^0_u$ & $\text{Order}$ & $E^1_v$ & $\text{Order}$ & $E^0_v$ & $\text{Order}$\\
		\midrule
		$1/5$ & 1.2902e-03 &       --       & 3.4797e-05 &      --        & 8.9389e-04 &     --         & 2.3487e-05 &    --          \\
		$1/10$ & 1.6027e-04 & 3.0090 & 2.1921e-06 & 3.9885 & 1.0555e-04 & 3.0821 & 1.4573e-06 & 4.0105 \\
		$1/15$ & 4.6642e-05 & 3.0443 & 4.3296e-07 & 4.0003 & 2.8985e-05 & 3.1875 & 2.8583e-07 & 4.0175\\
		$1/20$ & 1.9202e-05 & 3.0849 & 1.3690e-07 & 4.0024 & 1.1379e-05 & 3.2502 & 9.0140e-08 & 4.0116 \\
		$1/25$ & 9.5560e-06 & 3.1274 & 5.6090e-08 & 3.9986 & 5.5180e-06 & 3.2434 & 3.6948e-08 & 3.9967 \\
		\bottomrule
	\end{tabular}
\end{table}

\begin{table}
	\centering 
	\caption{Convergence rates for Example \ref{example2} in the temporal direction with simulation parameters \(\theta=0.35\) and \(k=1\).}
	\label{table9}
	\begin{tabular}{ccccc}
		\toprule 
		$\tau$  & $E^0_u$ & $\text{Order}$ & $E^0_v$ & $\text{Order}$\\
		\midrule
		$1/5$ &  6.6045e-02 &       --       &  3.9110e-02 &       --       \\
		$1/10$ &  1.5669e-02 & 2.0755 & 9.6759e-03 & 2.0151 \\
		$1/15$  & 6.9773e-03 & 1.9953 &  4.3292e-03 & 1.9835 \\
		$1/20$  & 3.9256e-03 & 1.9993 & 2.4411e-03 & 1.9915 \\
		$1/25$ &  2.5115e-03 & 2.0016 &  1.5638e-03 & 1.9958 \\
		\bottomrule
	\end{tabular}
\end{table}

\begin{table}
	\centering 
	\caption{Convergence rates for Example \ref{example2} in the temporal direction with simulation parameters \(\theta=0.65\) and \(k=2\).}
	\label{table10}
	\begin{tabular}{ccccccccc}
		\toprule 
		$\tau$ & $E^1_u$ & $\text{Order}$ & $E^0_u$ & $\text{Order}$ & $E^1_v$ & $\text{Order}$ & $E^0_v$ & $\text{Order}$\\
		\midrule
		$1/5$ & 2.4579e-01 &    --          & 4.7391e-02 &      --        & 6.9293e-02 &        --      & 4.1663e-03 &     --         \\
		$1/10$ & 6.2097e-02 & 1.9848 & 1.1892e-02 & 1.9947 & 1.7358e-02 & 1.9971 & 6.5323e-04 & 2.6731 \\
		$1/15$ & 2.7171e-02 & 2.0385 & 5.1634e-03 & 2.0575 & 7.7389e-03 & 1.9923 & 2.8765e-04 & 2.0229 \\
		$1/20$ & 1.5103e-02 & 2.0412 & 2.8543e-03 & 2.0605 & 4.3522e-03 & 2.0007 & 1.5143e-04 & 2.2302 \\
		$1/25$ & 9.6812e-03 & 1.9930 & 1.8300e-03 & 1.9920 & 2.7855e-03 & 1.9998 & 9.3968e-05 & 2.1385 \\
		\bottomrule
	\end{tabular}
\end{table}

\begin{figure}[htb]
	\centering
	\subfigure[\scriptsize $L^2$-norm errors for $u$]{
		\label{$L^2$-norm errors for u} 
		\includegraphics[width=2.5in]{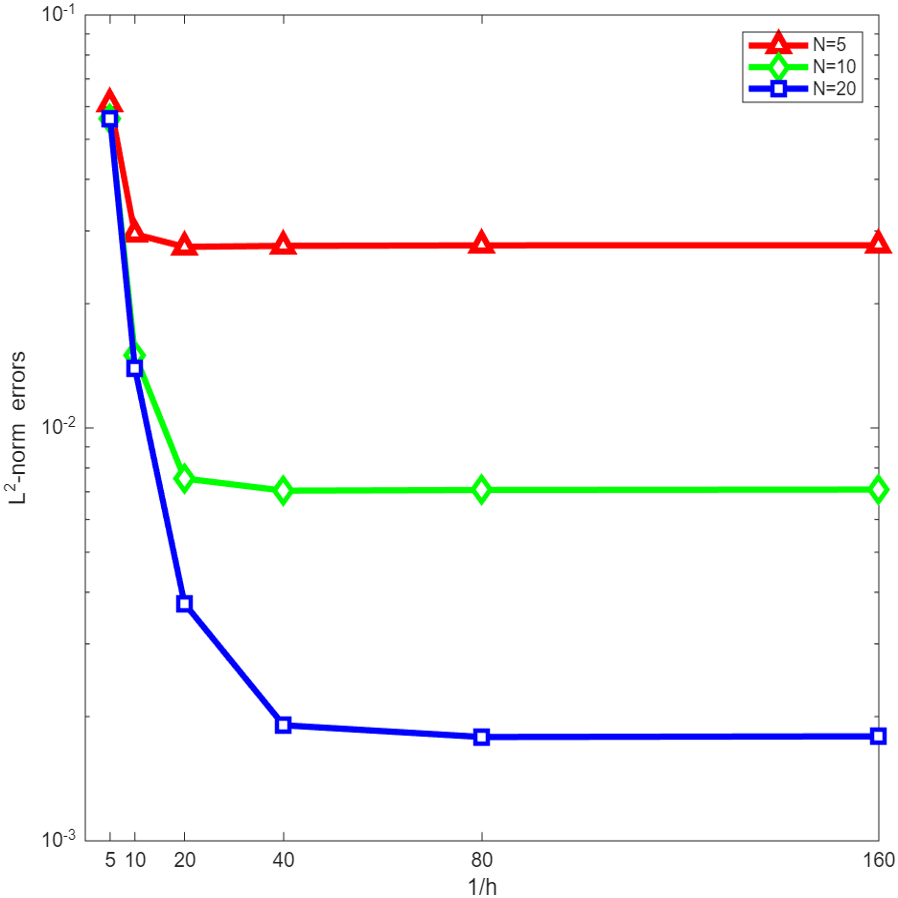}} 
	\hspace{20pt} 
	\subfigure[\scriptsize $L^2$-norm errors for $v$]{
		\label{$L^2$-norm errors for v} 
		\includegraphics[width=2.5in]{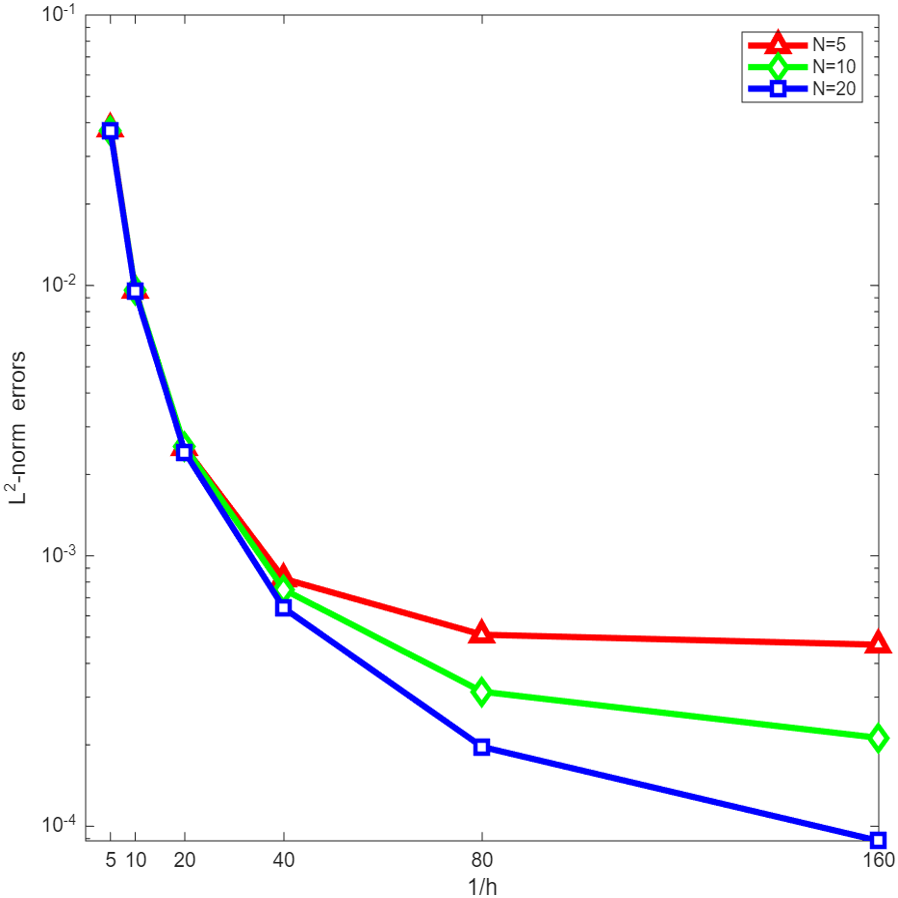}}
	\hspace{20pt}
	\caption{Trend of the $L^2$-norm errors of $u$ and $v$ as the mesh size $h$ decreases for Example \ref{example2} with parameters $k=1,\theta = 0.75$.}
	\label{figure2} 
\end{figure}

\begin{example}\label{example3}
In our final numerical example, we consider a three-dimensional problem governed by the equations given in \eqref{202601042211}-\eqref{202601042212}, with the corresponding exact solutions listed as follows:
\begin{align}
&u(x,y,z,t) = i\exp(t)\sin(\pi x)\sin(\pi y)\sin(\pi z),\quad v(x,y,z,t) = e^{it^2}\sin(x)(1-x)\sin(y)(1-y)\sin(z)(1-z).\notag
\end{align}

The purpose of this example is to verify the effectiveness of the numerical algorithm for solving high-dimensional problems. Tables \ref{table11} and \ref{table12} present the convergence accuracy in the spatial and temporal directions, respectively. These results are consistent with the theoretical analysis, demonstrating that the proposed method in this paper is capable of solving complex high-dimensional system.
\begin{table}
	\centering 
	\caption{Convergence rates for Example \ref{example3} in the spatial direction with simulation parameters \(\theta=0.5\) and \(k=1\).}
	\label{table11}
	\begin{tabular}{ccccccccc}
		\toprule 
		$h$ & $E^1_u$ & $\text{Order}$ & $E^0_u$ & $\text{Order}$ & $E^1_v$ & $\text{Order}$ & $E^0_v$ & $\text{Order}$\\
		\midrule
		$1/5$ & 7.5714e-01 &      --        & 4.1720e-02 &        --      & 1.1628e-02 &        --      & 6.4692e-04 &        --      \\
		$1/10$ & 3.8719e-01 & 0.9675 & 1.0880e-02 & 1.9390 & 5.9444e-03 & 0.9680 & 1.6790e-04 & 1.9460 \\
		$1/15$ & 2.5921e-01 & 0.9897 & 4.8737e-03 & 1.9806 & 3.9792e-03 & 0.9899 & 7.5139e-05 & 1.9829 \\
		$1/20$ & 1.9470e-01 & 0.9949 & 2.7490e-03 & 1.9904 & 2.9887e-03 & 0.9950 & 4.2368e-05 & 1.9916 \\
		\bottomrule
	\end{tabular}
\end{table}

\begin{table}
	\centering 
	\caption{Convergence rates for Example \ref{example3} in the temporal direction with simulation parameters \(\theta=0.5\) and \(k=1\).}
	\label{table12}
	\begin{tabular}{ccccc}
		\toprule 
		$\tau$  & $E^0_u$ & $\text{Order}$ &  $E^0_v$ & $\text{Order}$\\
		\midrule
	$1/5$ &  1.5229e-01 &     --         &  6.4111e-04 &   --           \\
	$1/10$ &  3.8759e-02 & 1.9742 &  1.4889e-04 & 2.1064 \\
	$1/15$ &  1.7331e-02 & 1.9851 &  6.5585e-05 & 2.0220 \\
	$1/20$& 9.7732e-03 & 1.9912 &  3.6749e-05 & 2.0135 \\
		\bottomrule
	\end{tabular}
\end{table}

\end{example}
\end{example}
\section{Conclusions}\label{section5}
In this paper, we construct a family of $\theta$-parameterized fully discrete finite element schemes, which include the standard Crank-Nicolson and leapfrog schemes as special cases.In addition, unlike the analytical and demonstrative approaches employed in existing literature, this paper proposes a relatively straightforward error analysis technique to establish the $L^{\infty}$-norm boundedness of the numerical solution. Future work will be pursued in two main directions. First, the proposed numerical method will be extended to polygonal mesh-based approaches, such as the Virtual Element Method (VEM), Weak Galerkin Finite Element Method (WG), and Hybrid High-Order (HHO) Method. Second, the error analysis techniques developed in this paper will be applied to more complex partial differential equations, including nonlinear time-dependent Joule heating equations and viscoelastic fluid models.

\end{document}